\documentclass[paper=a4, fontsize=11pt]{scrartcl}

\usepackage[T1]{fontenc}
\usepackage[utf8]{inputenc}
\usepackage{amsmath,amsfonts,amsthm,amssymb}
\usepackage[left=2cm,right=2cm,top=2.9cm,bottom=2.9cm]{geometry}

\usepackage[english]{babel}

\usepackage[Algorithm]{algorithm}
\usepackage{algorithmic}
\usepackage{graphicx}
\usepackage{setspace}

\DeclareMathOperator*{\argmin}{arg\,min}
\DeclareMathOperator{\Span}{span}

\usepackage{xcolor}
\numberwithin{equation}{section}

\newcommand{\norm}[1]{\left\lVert#1\right\rVert}

\newtheorem{theorem}{Theorem}[section]
\newtheorem{lemma}[theorem]{Lemma}

\newtheorem{definition}[theorem]{Definition}

\newtheorem{proposition}[theorem]{Proposition}

\usepackage{hyperref}
\usepackage{authblk}

\usepackage[german]{fancyref}
\providecommand{\keywords}[1]{\textbf{\textit{Key words ---}} #1}
\providecommand{\MSC}[1]{\textbf{\textit{MSC ---}} #1}
\DeclareMathAlphabet{\mathcall}{OMS}{cmsy}{m}{n}

\newcommand{\R}{\mathbb{R}}
\newcommand{\N}{\mathbb{N}}
\newcommand{\inner}[2]{\left\langle #1,\,#2\right\rangle}

\newcommand{\Hk}{\mathcall{H}_k\!\left(\Omega\right)}
\newcommand{\HkR}{\mathcall{H}_k\!\left(\mathbb{R}^N\right)}

\newcommand{\normW}[2]{ \left\|#1\right\|_{#2}}

\begin{document}
\selectlanguage{english}
  \title{Escaping the native space of Sobolev  kernels by interpolation}
  \author[1,2]{Tobias Ehring}
\author[1,3]{Max-Paul Vogel}
\author[1,3]{Bernard Haasdonk}
\affil[1]{{Institute of Applied Analysis and Numerical Simulation}, {University of Stuttgart}, {{Pfaffenwaldring 57}, {Stuttgart}, {70569}, {Baden-W\"urttemberg}, {Germany}}}
\affil[2]{Corresponding author. E-mail: \href{mailto:ehringts@mathematik.uni-stuttgart.de}{ehringts@mathematik.uni-stuttgart.de}}
\affil[3]{Contributing authors. E-mails: \href{max-paul.vogel@mathematik.uni-stuttgart.de}{max-paul.vogel@mathematik.uni-stuttgart.de},  \href{mailto:haasdonk@mathematik.uni-stuttgart.de}{haasdonk@mathematik.uni-stuttgart.de}}
  \date{\today}
  \maketitle
\begin{abstract}
\noindent Classical convergence analysis for kernel interpolation typically assumes that the target function $f$ lies in the reproducing kernel Hilbert space   $\Hk$ induced by a kernel on a domain $\Omega\subset\mathbb{R}^N$. For many applications, however, this assumption is overly restrictive.
We develop a general framework for analyzing the convergence of kernel interpolation {beyond the native space}. Let $A(\Omega)$ and $B(\Omega)$ be Banach spaces with continuous embeddings $\Hk \hookrightarrow A(\Omega)\hookrightarrow B(\Omega)$, assume point evaluation is continuous on $A(\Omega)$, and that $\Hk$ is dense in $A(\Omega)$. For a nested sequence of node sets $(X_n)_{n\ge1}\subset\Omega$ with $\bigcup_n X_n$ dense, we characterize convergence of the kernel interpolants in the $B(\Omega)$-norm for all target functions in $A(\Omega)$ via the uniform boundedness of the interpolation operators
$
\Pi^{\,n}_{A,B}:A(\Omega)\to B(\Omega)$.
This yields a necessary and sufficient condition under which kernel interpolation extends beyond $\Hk$.
Specializing to Sobolev kernels of order 
$\tau>N/2$ on bounded Lipschitz domains, we show that every 
$f \in C(\overline{\Omega})$ can be approximated in the
$L^2(\Omega)$-norm by interpolation using quasi-uniform nested centers.
Moreover, for a subclass of Sobolev kernels (including integer-order Matérn kernels), we prove that the Lebesgue constant is uniformly bounded on $[a,b]\subset\mathbb{R}$ under quasi-uniform centers; within our framework this implies supremum norm convergence of the interpolants for every target functions $f \in C([a,b])$. 
\end{abstract}

\noindent \keywords{kernel interpolation, scattered-data approximation, convergence beyond the native space, bounding the Lebesgue constant, reproducing kernel Hilbert spaces}\\

\noindent \MSC{41A05, 41A30,  46E22, 41A63 }

   \section{Introduction}\label{sec1}
   Kernel methods provide meshfree approximation schemes for multivariate functions from scattered data.
 Among these, kernel interpolation is straightforward to implement and, under a suitable sampling strategy, admits rigorous convergence guarantees; see \cite{TSH21}.  Classical error analysis is typically
developed under the assumption that the target function \(f\) belongs to the reproducing kernel Hilbert space (RKHS) \(\Hk\) (the native space) induced by a positive definite kernel \(k\) on \(\Omega \subset \R^N\). For \(f\in\Hk\), convergence in the \(\Hk\)-norm follows as the set of centers becomes dense in \(\Omega\); see Theorem~\ref{theo:convGeneral}. Moreover, convergence rates in weaker norms are available.
In the widely used case of Sobolev kernels of (possibly fractional) order \(\tau\) on a bounded Lipschitz domain \(\Omega \subset \R^N\), the space \(\Hk\) is norm-equivalent to the Sobolev space \(W_2^\tau(\Omega)\). If \(\tau > m + \tfrac{N}{2}\) for some \(m \in \N_0\), \(f \in \Hk\), and \(X_n \subset \Omega\) is a finite set of pairwise distinct centers with fill distance \(h_{X_n,\Omega}\) (see Section~\ref{sec:ScatterZero}), then for \(h_{X_n,\Omega}\) sufficiently small, the kernel interpolant \(s_f^n\) satisfies
\[
  \lvert f - s_f^n \rvert_{W_2^m(\Omega)}
  \;\le\;
  C_{\text{\tiny SZ}}'\, h_{X_n,\Omega}^{\,\tau - m}\, \|f\|_{\Hk}
  \quad\text{and}\quad
  \|f - s_f^n\|_{L^\infty(\Omega)}
  \;\le\;
  C_{\text{\tiny SZ}}'\, h_{X_n,\Omega}^{\,\tau - N/2}\, \|f\|_{\Hk},
\]
where \(C_{\text{\tiny SZ}}' > 0\) is independent of \(f\) and \(n\); see \cite[Corollary~11.33]{W04}. Here, $\lvert \, \cdot \, \rvert_{W_2^m(\Omega)}$ denotes the Sobolev space seminorm. 
From a practical standpoint, the assumption that the target function belongs to the RKHS \(\Hk\) is difficult to verify and, as numerical evidence indicates, often not essential for achieving accurate approximations. This motivates the central question: under what conditions on the target function, the error norm, the kernel, and the sampling strategy does kernel interpolation yield reliable approximations for functions outside the native space; in other words, when can one ``escape'' the native space?

\medskip
\noindent Existing literature has already addressed this question in the setting of Sobolev kernels where the native space of the kernel corresponds to \(W_2^{\tau}(\Omega)\) but the target function possesses only lower Sobolev regularity.
\noindent Let \(\Omega \subset \R^N\) be a bounded Lipschitz domain. Throughout the paper, this is a standing assumption. For Sobolev kernels of order \(\tau>0\) with \(\Hk \simeq W_2^{\tau}(\Omega)\), let \(X_n\subset\Omega\) be sets of pairwise distinct centers with fill distance \(h_{X_n,\Omega}\) and uniformly bounded mesh ratio (see Section~\ref{sec:ScatterZero}). Fix \(\beta\) with
$
  N/2 < \lfloor \beta \rfloor \le \beta \le \tau,
$
and \(0 \le \mu \le \beta\). Denote by \(s_f^{\,n}\) the kernel interpolant of \(f\) at \(X_n\). Then the following reduced-regularity error estimates hold:

\begin{itemize}
  \item[\textbf{(i)}] \cite[Thm.~4.2]{narcowich2006}, \cite[Thm.~4.2]{karvonen2020}.
  If \(f \in W_2^{\beta}(\Omega)\), then
  \begin{align*} 
    \| f - s_f^{\,n} \|_{W_2^{\mu}(\Omega)}
    \;\le\; C\, h_{X_n,\Omega}^{\,\beta - \mu}\, \| f \|_{W_2^{\beta}(\Omega)}, \qquad
    \| f - s_f^{\,n} \|_{L^{\infty}(\Omega)}
    \;\le\; C\, h_{X_n,\Omega}^{\,\beta - N/2}\, \| f \|_{W_2^{\beta}(\Omega)}.
  \end{align*}

  \item[\textbf{(ii)}] \cite[Thm.~3.10]{Narcowich2004}.
  If \(f \in C^{\beta}(\overline{\Omega})\) with \(\beta,\mu \in \N_0\), then
  \[
    | f - s_f^{\,n} |_{W_2^{\mu}(\Omega)}
    \;\le\;
    C'\, h_{X_n,\Omega}^{\,\beta-\mu}\, \|f\|_{C^{\beta}(\overline{\Omega})},
    \qquad
    \| f - s_f^{\,n} \|_{L^{\infty}(\Omega)}
    \;\le\;
    C'\, h_{X_n,\Omega}^{\,\beta-N/2}\, \|f\|_{C^{\beta}(\overline{\Omega})}.
  \]

  \item[\textbf{(iii)}] \cite[Thm.~4.2]{narcowich2004scattered}.
  For positive-definite kernels with polynomial spectral asymptotics (a class closely related to Sobolev kernels),   if \(f \in C_0^{\beta}(\R^N)\cap W_2^{\beta}(\R^N)\) with integer \(\beta\), and \(\alpha \in \N_0^N\) satisfying \(|\alpha| + N/2 < \beta\), then
  \[
    \| \partial^{\alpha} f - \partial^{\alpha} s_f^{\,n} \|_{L^{\infty}(\Omega)}
    \;\le\;
    C''\, h_{X_n,\Omega}^{\,\beta-|\alpha|-N/2}\, \max  \left\{ \|f\|_{C_0^{\beta}(\R^N) },\|f\|_{  W_2^{\beta}(\R^N)} \right\}.
  \]
\end{itemize}

\noindent Here \(C, C', C''>0\) are constants independent of \(f\) and \(n\) (but possibly depending on \(\Omega\), \(N\), \(\beta\), \(\tau\), the kernel, and the mesh ratio). These results above are typically proven via the standard band-limited approximation approach coupled with stability/inverse estimates; see the survey \cite{narcowich2005}. We also note that analogous estimates hold in the complementary case -- outside the scope of the present work -- where the target function is {smoother} than the RKHS, leading to various forms of superconvergence; see \cite{schaback1999improved,schaback2018superconvergence,karvonen2025generalsuperconvergencekernelbasedapproximation}.

\medskip
\noindent The estimates in \textbf{(i)} also extend to  manifolds: for Sobolev kernels restricted to a smooth, compact, embedded submanifold \(\mathcal{M}\subset\R^N\) without boundary, the same rates hold; see \cite[Thm.~17]{fuselier2012}. On the unit sphere \(\mathbb{S}^N\), analogous bounds are available for spherical basis function (SBF) interpolation with Sobolev kernels of order \(\tau\). If \(f\in W_2^{\beta}(\mathbb{S}^N)\), one obtains \(W_2^{\mu}(\mathbb{S}^N)\)-error estimates of the same order as in \textbf{(i)}; see \cite[Thm.~5.5]{narcowich2007direct} and \cite{Mhaskar2009}. Moreover, if \(f\in C^{2\beta}(\mathbb{S}^N)\) and \(\tau\ge 2\beta \ge N/2\), the SBF interpolant satisfies \(L^{\infty}(\mathbb{S}^N)\)-norm estimates as in \textbf{(ii)}; see \cite[Thm.~3.2]{Narcowich2002}.

\medskip

\noindent Related results hold for surface splines of order \(\tau\), which are conditionally positive definite kernels whose native space is the Beppo--Levi space \(BL^{\tau}(\Omega)\) (i.e., a Sobolev-type space modulo polynomials of degree \(\le \tau-1\)). For data sets \(X_n \subset \Omega\) that are \(\Pi_{\tau-1}\)-unisolvent, and target functions \(f \in BL^{\beta}(\Omega)\) with \(\tau \ge \beta > N/2\), one obtains \(L^{p}(\Omega)\)-error estimates for the corresponding interpolant for $1\leq p \leq \infty $; see \cite[Theorem~3.5]{brownlee2007}.

\medskip
\noindent A different interpolation strategy is analyzed in \cite{yoon2001spectral}, where the radial basis function shape parameter \(\gamma:=\gamma(h_{X_n,\Omega})\) is coupled to the vanishing fill distance \(h_{X_n,\Omega}\). For quasi-uniform data and target functions \(f\in W_\infty^{\beta}(\Omega)\) with integer \(\beta\ge 1\), the associated interpolant \(s_f^n\) satisfies
\[
\|f-s_f^n\|_{L^\infty(\Omega)}=O\!\left(h_{X_n,\Omega}^{\,\beta}\right)\quad\text{as }h_{X_n,\Omega}\to 0.
\]
This holds for standard choices including the multiquadric kernel, the inverse multiquadric kernel, and the shifted surface splines. This applies to standard kernels including the multiquadric kernel, the inverse multiquadric kernel, and shifted surface splines. A related rescaled, localized radial basis function scheme is  investigated in \cite{DeMarchi2020}. There, for Sobolev kernels induced by compactly supported radial basis functions, the shape parameter is chosen proportional to the fill distance $h_{X_n,\Omega}$.  For a bounded Lipschitz domain $\Omega \subset \R^N$ and quasi-uniform interpolation nodes, this leads to the following convergence estimate: there exists a constant $C''' > 0$ such that
\[
  \|f - s_f^{\,n}\|_{L^{\infty}(\Omega)}
  \;\le\;
  C''' \, h_{X_n,\Omega}\, \|f\|_{C^{1}(\Omega)}
  \quad \text{for all } f \in C^{1}(\Omega).
\]
\noindent  For kernel-based quadrature, convergence extends also beyond the native-space setting: even when the integrand does not belong to \(\Hk\), one can derive rates under Sobolev smoothness assumptions on the integrand (e.g., \(f\in W_2^{\beta}(\Omega)\)), with the rates adapting to \(\beta\); see \cite{kanagawa2016,kanagawa2018}.

\medskip
\noindent With the exception of the adapted shape-parameter scheme in \cite{yoon2001spectral}, the preceding results assume target smoothness strictly above the critical index \(N/2\) (e.g., \(f\in W_2^{\beta}(\Omega)\) or \(f\in C^{\beta}(\overline{\Omega})\) with \(\beta>N/2\)). From an applied perspective, the regime of merely continuous target functions \(f\in C(\overline{\Omega})\) is especially relevant and will be a recurring focus. In this setting, the question of “escaping” the native space is particularly compelling for {universal} kernels, i.e., continuous kernels whose RKHS is dense in \(C(\overline{\Omega})\) in the uniform norm; see \cite{steinwart2002}. A canonical example is given by Sobolev kernels of order \(\tau\) on bounded Lipschitz domains with \(\lfloor\tau\rfloor>N/2\). While universality guarantees the {existence} of RKHS approximants for every \(f\in C(\overline{\Omega})\), it does not, by itself, provide a constructive, data-driven scheme (such as interpolation) with convergence guarantees. This motivates our study of native-space escape for merely continuous target functions on $\overline{\Omega}$. However, in this case, no general convergence {rates} should be expected; indeed, one can conversely infer target smoothness from observed rates for Sobolev kernels, see \cite{Wenzel2025}.

\medskip
\noindent \textbf{Main contributions.}
  We develop a general framework for analyzing {native-space escape by interpolation}. Given function spaces \(A(\Omega)\) (for the target functions) and \(B(\Omega)\) (for the error metric),  we consider the family of linear interpolation operators
$ 
\Pi^{\,n}_{A,B}:A(\Omega)\to B(\Omega),
$ which will be specified explicitly in Section~\ref{sec3}.
Uniform boundedness of \(\{\|\Pi^{\,n}_{A,B}\|_{\mathcal{L}(A(\Omega),B(\Omega))}\) is shown to be a decisive criterion ensuring that interpolation provides reliable approximation in the \(B(\Omega)\)-norm for target functions in \(A(\Omega)\), even when these target functions lie outside the native RKHS. This recovers and streamlines a result from \cite[Corollary~4.3]{narcowich2006}: for quasi-uniform center sets \(X_n\), the interpolation operator $$\Pi^{\,n}_{W_2^{\beta}(\Omega),W_2^{\beta}(\Omega)}:W_2^{\beta}(\Omega)\to W_2^{\beta}(\Omega)$$ is uniformly bounded. Within our framework, this immediately yields native-space escape from \(\Hk\simeq W_2^{\tau}(\Omega)\) to \(A(\Omega)=B(\Omega)=W_2^{\beta}(\Omega)\) for \(\beta\le\tau\).

\medskip
\noindent We further apply the framework to Sobolev kernels with continuous target functions on $\overline{\Omega}$, with the approximation error measured in the \(L^2(\Omega)\)-norm. Under natural uniform-sampling assumptions, we prove uniform boundedness of the interpolation operators, yielding native-space escape in the  \(L^2(\Omega)\)-norm; to the best of our knowledge, this has not been shown yet.

\medskip
\noindent Finally, for Sobolev kernels  and the continuous target functions  on $\overline{\Omega}$ in the supremum norm setting, one has \(\|\Pi^{\,n}_{C,C}\|=\Lambda_{X_n}\), the Lebesgue constant. On compact, boundaryless manifolds, we have \(\sup_n\Lambda_{X_n}<\infty\) for quasi-uniform \(X_n\) and integer orders \(\tau>N/2\); see \cite[Thm.~4.6]{HNW10}. This does not directly extend to Euclidean domains with boundary: on bounded \(\Omega\subset\R^N\), \cite{marchi2010} shows \(\Lambda_{X_n}\lesssim \sqrt{|X_n|}\) for finitely smooth radial basis functions, while the numerical experiments in that work even suggest uniform boundedness. As a new positive result in the presence of a boundary, we prove that on intervals \([a,b]\subset\R\) and for Sobolev kernels of suitable order, \(\{\Lambda_{X_n}\}_n\) is uniformly bounded for quasi-uniform nodes; consequently, interpolation achieves native-space escape in the supremum norm for all \(f\in C([a,b])\).

\medskip
\noindent The remainder of the paper is organized as follows. Section~\ref{sec2} gives a background on kernel interpolation with emphasis on Sobolev kernels. Section~\ref{sec3} develops the general framework for native-space escape by interpolation. Section~\ref{sec4} applies this framework to Sobolev kernels with continuous target functions on $\overline{\Omega}$ and approximation error measured in the \(L^2\)-norm. In Section~\ref{sec5} the \(L^2\)-norm error is replaced by the stronger supremum-norm error. Section~\ref{sec6} presents conclusions and directions for future work.

 \section{Background on kernel interpolation}\label{sec2}

We briefly recall basic notions on positive definite kernels, RKHSs, and kernel interpolation.
Let $\Omega\subset\R^N$ be nonempty. A {kernel} is any symmetric function
$
  k:\Omega\times\Omega\to\R. 
$
For a finite set of pairwise distinct points $X_n = \{x_1,\dots,x_n\}\subset\Omega$, the associated {Gramian  matrix} $K_{X_n} \in \R^{n \times n}$ is defined by
\[
  (K_{X_n})_{ij} \;=\; k(x_i,x_j),\qquad i,j=1,\dots,n.
\]
The kernel $k$ is called {positive definite} (p.d.) if
$
  z^\top K_{X_n} z \,\geq\, 0
$
for every finite set of pairwise distinct points $X_n \subset \Omega$ and all $z \in \mathbb{R}^n$, and it is called {strictly positive definite} (s.p.d.) if
$
  z^\top K_{X_n} z \,>\, 0
$
for every finite set of pairwise distinct points $X_n \subset \Omega$ and all nonzero $z \in \mathbb{R}^n$.
 \\
 Every p.d.\! kernel $k$ induces a unique RKHS $\Hk$ of real-valued functions on $\Omega$ with reproducing kernel $k$ (Moore--Aronszajn theorem, e.g.\ \cite[Thm.~10.10]{W04}). An RKHS is a Hilbert space of functions \(f:\Omega\to\mathbb{R}\) such that the point‐evaluation functionals
\[
  \delta_x:\Hk\to\mathbb{R},\quad \delta_x(f):=f(x),
\]
are continuous for all \(x\in\Omega\).  Equivalently (see proof of Theorem 10.2 in \cite{W04}), there exists a kernel \(k\) (reproducing kernel) satisfying
$
  k(x,\cdot)\in\Hk$
and
\begin{align}\label{eq:reproducing}
  \langle f,\,k(x,\cdot)\rangle_{\Hk}=f(x)
  \quad\forall\,f\in\Hk,\ \forall\,x\in\Omega ,
\end{align}
known as the {reproducing property}. If $k$ is continuous on $\Omega\times\Omega$, then the map $x\mapsto k(x,\cdot)\in\Hk$ is continuous, and every $f\in\Hk$ is continuous on $\Omega$. This is a consequence of the reproducing property, since, for all $x,y \in \Omega$, 
\begin{align*}
    |f(x)-f(y)| = &  |\langle f,\,k(x,\cdot)-k(y,\cdot)\rangle_{\Hk}| \\
    \leq& \| f\|_{\Hk} \| k(x,\cdot)-k(y,\cdot) \|_{\Hk}  
    = \| f\|_{\Hk} \sqrt{\left(k(x,x)-2k(x,y)+k(y,y)\right)}.
\end{align*}
\noindent For a given p.d.\! kernel $k:\Omega\times\Omega\to\R$,  fixed pairwise distinct {nodes} $X_n=\{x_1,\dots,x_n\}\subset\Omega$ (also called centers in this context), and a target function $f:\Omega\to\R$, the  minimal-norm interpolation problem in $\Hk$ is defined as
\begin{equation}\label{eq:rkhs-interp}
  s_f^n \;=\; \argmin_{s\in\Hk}\Bigl\{ \norm{s}_{\Hk} \;\big|\; s(x_i)=f(x_i)\ \text{for }i=1,\dots,n \Bigr\}.
\end{equation}
Although \eqref{eq:rkhs-interp} is posed in an infinite-dimensional space,  its solution lies in the finite-dimensional subspace
\[
  V(X_n)\;:=\; \Span\bigl\{\,k(x_i,\cdot):\,i=1,\dots,n\,\bigr\}\subset\Hk
\]
and admits the representation
\begin{equation}\label{eq:representer}
  s_f^n(x) \;=\; \sum_{i=1}^n \alpha_i\,k(x_i,x),\qquad \alpha=(\alpha_1,\dots,\alpha_n)^\top\in\R^n,
\end{equation}
where the coefficient vector $\alpha$ solves the linear system
\begin{equation}\label{eq:gram-system}
  K_{X_n}\,\alpha \;=\; r,\qquad r:=\bigl(f(x_1),\dots,f(x_n)\bigr)^\top.
\end{equation}
Indeed, if $k$ is s.p.d., then $K_{X_n}$ is invertible, hence \eqref{eq:rkhs-interp} has a unique solution for {any} data vector $r\in\R^n$ (in particular, $f$ is not required to be in $\Hk$). Furthermore, the RKHS norm of the interpolant can be expressed as
\begin{equation}\label{eq:minimalNorm}
  \norm{s_f^n}_{\Hk}^2 \;=\; r^\top K_{X_n}^{-1} r \;=\; \alpha^\top K_{X_n}\alpha.
\end{equation}
An alternative formulation of the interpolant avoids repeated inversion of the Gramian  matrix when fitting multiple target functions.  One introduces the so‐called Lagrange functions, which form a cardinal basis of $V(X_n)$, i.e.
there exists a unique family $\{l_{x_i}\}_{i=1}^n\subset V(X_n)$ with the cardinal  property
\begin{align}\label{eq:LagrangBasis}
    l_{x_i}(x_j)=\delta_{ij}\qquad (i,j=1,\dots,n).
\end{align}
With this basis the interpolant takes the {Lagrange form}
\begin{equation}\label{eq:lagrange-rep}
  s_f^n(x) \;=\; \sum_{i=1}^n f(x_i)\,l_{x_i}(x),\qquad x\in\Omega.
\end{equation}
Using this we define the interpolation operator
\begin{equation}\label{eq:Pi_Hk}
  \Pi^n:\Hk\to\Hk,\qquad
  \Pi^n f=\sum_{i=1}^{n} f(x_i)\,l_{x_i}(\cdot).
\end{equation}
Note that $\Pi^n$ is the $\Hk$-orthogonal projection onto the closed subspace $V(X_n)\subset \Hk$, since for any $f\in\Hk$ and any $v=\sum_{i} a_i\, k(x_i,\cdot)\in V(X_n)$,
\begin{align}\label{eq:acrossTErmZEros}
  \langle f-\Pi^n f,\,v\rangle_{\Hk}
  \;=\; \sum_i a_i\bigl(f(x_i)-\left(\Pi^n f\right)(x_i)\bigr)\;=\;0
\end{align}
by the reproducing property \eqref{eq:reproducing}, so $f-\Pi^n f\perp V(X_n)$. Since $\Pi^n$ is an orthogonal projection with nontrivial range $V(X_n)$ and nontrivial orthogonal complement $V(X_n)^\perp$, it follows that
\begin{align}\label{eq:operotOne}
\Vert \Pi^n \Vert_{\mathcal{L}(\Hk)}=1 \quad \text{and}\quad \Vert I-\Pi^n \Vert_{\mathcal{L}(\Hk)}=1.
\end{align}
Furthermore, under the assumption $f\in\Hk$ we obtain convergence in the RKHS norm provided the centers become dense in $\Omega$.
\begin{theorem}\label{theo:convGeneral}
Let $\Omega\subset\R^{N}$ be a domain and let $k:\Omega\times\Omega\to\R$ be a continuous, s.p.d.\! kernel with RKHS $\Hk$. Let $(X_n)_{n\in\N}$ be a nested family of finite, pairwise distinct point sets $X_n\subset\Omega$ with dense union,
\[
  X_1\subset X_2\subset\cdots,\qquad \overline{\bigcup_{n\in\N} X_n}=\overline{\Omega},
\]
and let $f\in \Hk$. Then
\[
  \lim_{n\to\infty} \norm{f-\Pi^n f }_{\Hk} =0.
\]
\end{theorem}
\begin{proof}
Define $a_n:=\lVert s_f^n  \rVert_{\Hk}$. Since $X_n\subset X_{n+1}$, the feasible set in \eqref{eq:rkhs-interp} for $n{+}1$ is contained in that for $n$, hence $a_n\le a_{n+1}$ for all $n$. Moreover, $f$ is feasible for every $n$, so $a_n\le \norm{f}_{\Hk}$. Thus $\left(a_n\right)_{n \in \mathbb{N}}$ is monotonically increasing and bounded, hence convergent.\\
For $m>n$, write $s_f^m=\Pi^m f$ and $s_f^n=\Pi^n f$. Because $V(X_n)\subset V(X_m)$, we have $\Pi^n\Pi^m=\Pi^n$ and since orthogonal projections are self-adjoint, it follows
\[
  \inner{(\Pi^m-\Pi^n)f}{\Pi^n f}_{\Hk}
  \;=\; \inner{f}{(\Pi^m-\Pi^n)\Pi^n f}_{\Hk}
  \;=\; \inner{f}{\Pi^n f-\Pi^n f}_{\Hk}
  \;=\; 0,
\]
so $(\Pi^m-\Pi^n)f \perp \Pi^n f$. Hence, by the Pythagoras identity,
\[
  \norm{s_f^m}_{\Hk}^2
  \;=\; \norm{s_f^n}_{\Hk}^2 + \norm{s_f^m-s_f^n}_{\Hk}^2 \;
\Longrightarrow \; \norm{s_f^m-s_f^n}_{\Hk}^2=a_m^2-a_n^2\to 0\] as $m,n\to\infty$. Thus $\bigl(s_f^n\bigr)_{n \in \mathbb{N}}$ a is Cauchy sequence in $\Hk$, and completeness yields a limit $s_f:=\lim_{n\to\infty}s_f^n$.\\
For any $x\in\bigcup_{n} X_n$, choose $m$ with $x\in X_m$, whence $s_f^{m}(x)=f(x)$ and therefore $s_f(x)=\lim_{n\to\infty}s_f^{n}(x)=f(x)$. Since $k$ is continuous, every element of $\Hk$ is continuous; hence, in particular, 
$s_f$ and $f$ are continuous. Because $\bigcup_{n}X_n$ is dense in $\Omega$, it follows that $s_f=f$ on $\Omega$. Consequently,
\[
  \lim_{n\to\infty}\norm{f-\Pi^n f}_{\Hk}
  \;=\; \lim_{n\to\infty}\norm{s_f-s_f^n}_{\Hk}
  \;=\; 0.
\]
\end{proof}

\noindent The minimal-norm interpolation problem is well-posed and admits a unique solution, which -- under pairwise distinct nodes -- can be written in the Lagrange form \eqref{eq:lagrange-rep}, even when the data are samples of a target function $f\notin\Hk$. We next record a criterion that characterizes membership $f\in\Hk$ via the boundedness of the associated sequence of minimal-norm interpolants.
\begin{theorem}\label{thm:bounded-interpolants}
Let $\Omega\subset\R^{N}$ be a domain and let $k:\Omega\times\Omega\to\R$ be a continuous, s.p.d.\! kernel with RKHS $\Hk$. Let $(X_n)_{n\in\N}$ be a nested family of finite, pairwise distinct point sets $X_n\subset\Omega$ with dense union,
\[
  X_1\subset X_2\subset\cdots,\qquad \overline{\bigcup_{n\in\N} X_n}=\overline{\Omega}.
\]
For $f\in C(\Omega)$, let $\left(s_f^n \right)_{n \in \mathbb{N}} \subset\Hk$ be the sequence of  unique minimal-norm interpolants of $f|_{X_n}$. Then
\[
  \sup_{n\in\N}\norm{s_f^n}_{\Hk}<\infty \quad\Longleftrightarrow\quad f\in\Hk.
\]
\end{theorem}

\begin{proof}
($\Leftarrow$) If $f\in\Hk$, then $f$ is feasible for each interpolation problem, and minimality yields $\|s_f^n\|_{\Hk}\le \norm{f}_{\Hk}$ for all $n$.

\medskip
\noindent ($\Rightarrow$) Suppose $\sup_n \|s_f^n\|_{\Hk}<\infty$. Since $\Hk$ is a Hilbert space, the sequence of interpolants is weakly relatively compact; there exist a subsequence $(s_f^{n_l})$ and $s\in\Hk$ with $s_f^{n_l}\rightharpoonup s$ in $\Hk$. Fix $x$ in the dense set $\bigcup_n X_n$. Then $s_f^{n_l}(x)=f(x)$ for all sufficiently large $l$. Using the reproducing property \eqref{eq:reproducing} and weak convergence,
\[
  f(x) \;=\; \lim_{l\to\infty}\inner{s_f^{n_l}}{k(x,\cdot)}_{\Hk}
  \;=\; \inner{s}{k(x,\cdot)}_{\Hk}
  \;=\; s(x).
\]
Hence $s=f$ on a dense set. Because $k$ is continuous, every function from $\Hk$ is continuous on $\Omega$, so $s,f\in C (\Omega)$ and continuity implies $s=f$ on $\Omega$. Thus $f\in\Hk$.
\end{proof}
\noindent The previous result provides a verification condition for the RKHS membership that can be examined numerically. Specifically, compute \eqref{eq:minimalNorm} along a nested sequence of node sets $(X_n)_{n\in\N}$ whose union is dense in the ambient domain $\Omega$. If the sequence $\|s_f^n\|_{\Hk}$ remains bounded, then $f\in\Hk$; if it diverges, then $f\notin\Hk$. In many practical applications one observes divergence, indicating that the underlying function does not belong to the RKHS, but also observes pointwise convergence of the minimal-norm interpolants. This is precisely the question of whether we can escape the native space via interpolation.
Before analyzing this phenomenon in detail, we introduce the kernel classes most relevant for the remainder of this work.

\subsection{Translation-invariant kernels with algebraic spectral decay}
\label{subsec:ti-algebraic}
We study translation-invariant kernels on \(\R^{N}\) whose spectral densities decay algebraically. A kernel \(k\colon\R^{N}\times\R^{N}\to\R\) is called  {translation-invariant} if there exists \(\Phi\colon\R^{N}\to\R\) with \(k(x,y)=\Phi(x-y)\) for all \(x,y\in\R^{N}\). Assume \(\Phi\in C(\R^{N})\cap L^{1}(\R^{N})\) is real-valued and positive definite (in the sense of \cite{W04}),  then \(k\) is s.p.d.\! and its  associated RKHS on \(\R^{N}\) admits the Fourier characterization
\begin{equation}\label{eq:FourierRKHS}
  \HkR \;=\; \Bigl\{\, f\in L^{2}(\R^{N})\cap C(\R^{N}) \ \Bigm|\ 
  \widehat{f} / \sqrt{\widehat{\Phi}}\in L^{2}(\R^{N}) \Bigr\},
\end{equation}
with inner product
\begin{equation}\label{eq:FourierInner}
  \inner{f}{g}_{\HkR}
  \;=\; (2\pi)^{-N/2}\int_{\R^{N}}
  \frac{\widehat{f}(\omega)\,\overline{\widehat{g}(\omega)}}{\widehat{\Phi}(\omega)}\,\mathrm{d}\omega,
  \qquad f,g\in\HkR,
\end{equation}
where \(\widehat{\cdot}\) denotes the Fourier transform in the convention consistent with \eqref{eq:FourierInner} (see \cite[Thm.~10.12]{W04}). A notable subclass is given by those \(\Phi\) whose spectra satisfy two-sided algebraic decay: there exist \(c_{\Phi},C_{\Phi}>0\) and a  \(\tau>\tfrac N2\) such that
\begin{equation}\label{eq:algebDecay}
  c_{\Phi}\,(1+\|\omega\|_{2}^{2})^{-\tau}
  \;\le\; \widehat{\Phi}(\omega) \;\le\;
  C_{\Phi}\,(1+\|\omega\|_{2}^{2})^{-\tau}
  \quad \forall\,\omega\in\R^{N}.
\end{equation}
Kernels satisfying \eqref{eq:algebDecay} are s.p.d.\! (cf.\ \cite[Thm.~6.11]{W04}), and their native space $\HkR$ coincides, up to equivalent norms, with the Bessel-potential space
\begin{equation}\label{eq:BesselSobolev}
  H^\tau\!\left(\R^{N}\right) \;=\; \Bigl\{\, f\in L^{2}(\R^{N}) \ \Bigm|\ 
  (1+\|\cdot\|_{2}^{2})^{\tau/2}\,\widehat{f}(\cdot)\in L^{2}(\R^{N}) \Bigr\},
\end{equation}
see \cite[Cor.~10.13]{W04}. Bessel-potential spaces are also RKHSs corresponding to the translation-invariant kernel with
\[
  \widehat{\Phi}(\omega)=(1+\|\omega\|_{2}^{2})^{-\tau},\qquad \tau>\tfrac N2
\]
whose kernels have  the closed form
\begin{equation}\label{eq:MaternGeneral}
  k(x,y)
  \;=\; \frac{2^{\,1-\tau}}{(2\pi)^{N/2}\Gamma(\tau)}\,\|x-y\|_{2}^{\,\tau-\frac{N}{2}}\,
  K_{\tau-\frac{N}{2}}\!\bigl(\|x-y\|_{2}\bigr),\qquad x,y\in\R^{N},
\end{equation}
where \(K_{\nu}\) denotes the modified Bessel function of the second kind; see \cite[Thm.~6.13]{W04}. These are the Matérn kernels with smoothness \(\nu:=\tau-\tfrac N2>0\). For half-integer \(\nu\) (equivalently, \(\tau-\tfrac N2\in\tfrac12\N\)) one obtains elementary forms up to a positive constant (cf.\ \cite[p.~41]{fasshauer2007}):
\begin{align}
  k(x,y) &= c_{N,\nu}e^{-\|x-y\|_{2}} 
    && \text{for } \tau=\tfrac{N+1}{2}\;(\nu=\tfrac12), \label{eq:matern:12}\\
  k(x,y) &= c_{N,\nu}\bigl(1+\|x-y\|_{2}\bigr)\,e^{-\|x-y\|_{2}}
    && \text{for } \tau=\tfrac{N+3}{2}\;(\nu=\tfrac32), \label{eq:matern:32}\\
  k(x,y) &= c_{N,\nu}\bigl(3+3\|x-y\|_{2}+\|x-y\|_{2}^{2}\bigr)\,e^{-\|x-y\|_{2}}
    && \text{for } \tau=\tfrac{N+5}{2}\;(\nu=\tfrac52). \label{eq:matern:52}
\end{align}
The Bessel-potential space agrees, as a set and up to equivalent norms, with the (possibly fractional) Sobolev space \(W_2^\tau\left(\mathbb{R}^N\right)\) defined via weak derivatives: writing \(\tau=r+s\) with \(r=\lfloor \tau\rfloor\in\N_{0}\) and \(s\in[0,1)\), one has
\begin{equation}\label{eq:SObolevSpace}
  W_2^\tau\left(\mathbb{R}^N\right) \;=\; \Bigl\{\, f\in L^{2}(\R^{N}) \ \Bigm|\ 
   \partial^{\alpha} f \in L^{2}(\R^{N}) \ \text{for all } |\alpha|\le r \ \text{ and } \ |f|_{W_2^{r+s}\left(\mathbb{R}^N\right)}<\infty \Bigr\},
\end{equation}
with seminorm
\begin{align*}
   |f|_{W_2^m(\mathbb{R}^N)}^{2} \;:=\;
  \sum_{|\alpha|=m}\|\partial^{\alpha}f\|_{L^{2}(\R^{N})}^{2},  
\end{align*}
 Slobodeckij seminorm (interpreted as \(0\) when \(s=0\))
\begin{equation*}
    |f|_{W_2^{r+s}\left(\mathbb{R}^N\right)}^{2}
    \;:=\; \sum_{|\alpha|=r}\int_{\R^{N}}\!\int_{\R^{N}}
    \frac{\bigl|\partial^{\alpha}f(x)-\partial^{\alpha}f(y)\bigr|^{2}}{\|x-y\|_{2}^{\,N+2s}}\,
    \mathrm{d}x\,\mathrm{d}y,
\end{equation*}
and the norm
\begin{equation}\label{eq:SobolevDef}
  \|f\|_{W_2^{r+s}(\R^N)}^{2} \;:=\;
  \sum_{m=0}^r |f|_{W_2^m(\mathbb{R}^N)}^{2} 
  \;+\; |f|_{W_2^{r+s}\left(\mathbb{R}^N\right)}^{2},
\end{equation}
see, e.g., \cite[Thm.~3.16]{mclean2000strongly}. We shall refer to translation-invariant kernels on $\R^N$ whose spectral densities satisfy \eqref{eq:algebDecay} as  {Sobolev kernels of order \(\tau\)}. \\

\noindent Generally, if $k$ is an s.p.d.\! kernel on $\R^{N}$ and $\Omega\subset\R^{N}$ is nonempty, then the restriction $k|_{\Omega\times\Omega}$ is again s.p.d. By Aronszajn’s seminal work \cite{aronszajn1950theory}, its RKHS is
\begin{align}\label{eq:restRKHS1}
  \Hk \;=\; \{\, f|_{\Omega} : f\in \HkR \,\},
\end{align}
endowed with the minimal-extension norm
\begin{align}\label{eq:restRKHS2}
  \|u\|_{\Hk}
  \;=\;
  \inf\bigl\{\, \|g\|_{\HkR} \ \bigm|\ g\in \HkR,\ g|_{\Omega}=u \,\bigr\},
\end{align}
and the infimum is achieved by a unique minimal norm extension $g_{u}\in\HkR$ of $u$. \\
For Sobolev spaces, a  restricted version corresponding to the global space in \eqref{eq:SObolevSpace} can be defined by 
\begin{equation}\label{eq:SObolevSpaceLocal}
  W_2^\tau\left(\Omega\right) \;=\; \Bigl\{\, f\in L^{2}(\Omega) \ \Bigm|\ 
   \partial^{\alpha} f \in L^{2}(\Omega) \ \text{for all } |\alpha|\le r \ \text{ and } \ |f|_{W_2^{r+s}\left(\Omega\right)}<\infty \Bigr\},
\end{equation}
with 
\begin{align}
   |f|_{W_2^m(\Omega)}^{2} \;&:=\;
  \sum_{|\alpha|=m}\|\partial^{\alpha}f\|_{L^{2}(\Omega)}^{2}, \notag \\
     |f|_{W_2^{r+s}\left(\Omega\right)}^{2}
    \;&:=\; \sum_{|\alpha|=r}\int_{\Omega}\!\int_{\Omega}
    \frac{\bigl|\partial^{\alpha}f(x)-\partial^{\alpha}f(y)\bigr|^{2}}{\|x-y\|_{2}^{\,N+2s}}\,
    \mathrm{d}x\,\mathrm{d}y, \notag \\
   \|f\|_{W_2^{r+s}(\Omega)}^{2} \;&:=\;
  \sum_{m=0}^r |f|_{W_2^m(\Omega)}^{2} 
  \;+\; |f|_{W_2^{r+s}\left(\Omega\right)}^{2}. \label{eq:localSobNorm}
\end{align}
Following the discussion in the proof of Lemma~3.1 in \cite{Narcowich2004}, if \(\Omega\subset\R^{N}\) is a bounded Lipschitz domain and \(\tau>\frac{N}{2}\), there exists a continuous extension operator \(\mathcal{E}\colon W_2^\tau(\Omega)\to W_2^\tau(\R^{N})\) with \(\left(\mathcal{E} f\right)|_{\Omega}=f\) for all \(f\in W_2^\tau(\Omega)\) and a constant \(C_{\mathcal{E}}>0\) such that
\[
    \| \mathcal{E} f \|_{W_2^\tau\left(\R^{N}\right)} \;\le\; C_{\mathcal{E}}\; \| f \|_{W_2^\tau\left(\Omega\right)}.
\]
Combining this with the argument in  the proof of Corollary~10.48 in \cite{W04} shows that, for a Sobolev kernel with \(\tau>\tfrac{N}{2}\) and bounded Lipschitz domain \(\Omega \subset \R^N\), the RKHS \(\Hk\) coincides with \(W_2^\tau(\Omega)\) with equivalent norms, i.e., there exist constants \(c_{\tau,\Omega},C_{\tau,\Omega}>0\) such that
\begin{equation}\label{eq:NormEquiLocal}
  c_{\tau,\Omega}\,\|u\|_{W_2^{\tau}(\Omega)} \;\le\; \|u\|_{\Hk} \;\le\; C_{\tau,\Omega}\,\|u\|_{W_2^{\tau}(\Omega)}
  \qquad \forall\,u\in W_2^{\tau}(\Omega).
\end{equation}
As a consequence of the norm equivalences on bounded Lipschitz domains and on the whole space \(\R^N\) (for \(\tau>\tfrac N2\)), the spaces \(W^{\tau}_2(\mathbb{R}^N)\) and \(W^{\tau}_2(\Omega)\) are themselves RKHSs, endowed with the norms \eqref{eq:SobolevDef}   and \eqref{eq:localSobNorm}, respectively. On \(\R^{N}\), the reproducing kernel is translation-invariant with spectral density proportional to \((1+\|\omega\|_{2}^{2})^{-\tau}\) (see, e.g., \cite{novak2018reproducing}). On bounded Lipschitz domains, a reproducing kernel still exists, but in general it is {not} translation-invariant. As a one-dimensional illustration, consider \(W^1_2(a,b)\) endowed with
\[
  \inner{u}{v}_{W^1_2(a,b)} \;=\; \int_{a}^{b} u\,v\,\mathrm{d}x \;+\; \int_{a}^{b} u'\,v'\,\mathrm{d}x .
\]
Indeed, the point evaluation is continuous on \(W^1_2(a,b)\); for any \(u\in W^1_2(a,b)\) and \(x\in[a,b]\),
\[
  |u(x)| \;\le\; \frac{1}{\sqrt{b-a}}\,\|u\|_{L^{2}(a,b)}
  \;+\; \sqrt{b-a}\,\|u'\|_{L^{2}(a,b)}
  \;\le\; \left(\frac{1}{\sqrt{b-a}}+\sqrt{(b-a)}\right)\,\|u\|_{W^1_2(a,b)}  
\]
and the associated reproducing kernel is
\begin{align}\label{eq:reproH1}
  k(x,y) \;=\; \frac{1}{\sinh(b-a)}
  \begin{cases}
    \cosh(b-y)\,\cosh(x-a), & x\le y,\\[2pt]
    \cosh(b-x)\,\cosh(y-a), & x>y,
  \end{cases}
\end{align}
obtained by a brief integration-by-parts argument and hyperbolic addition formulas. This shows how the effect of bounded domain destroy translation invariance even when the underlying Sobolev space is an RKHS; such kernels are therefore not Sobolev kernels in the sense of \eqref{eq:algebDecay}. \\

\noindent We next introduce a subclass of Sobolev kernels of integer order that admits a decomposition of the global RKHS norm on $\R^N$ into a sum of local contributions associated with a finite open partition of $\R^N$. This structural property will be crucial in the sequel. 
\begin{definition}[Sobolev kernel with local norm decomposition]\label{def:SObolevWIthLocalNOrm}
Let $k$ be a Sobolev kernel of integer order $\tau \in \mathbb{N}$. We say that $k$ admits a local norm decomposition if, for every open set $\Omega \subset \R^{N}$, there exists a norm $\normW{\cdot}{\mathcal{G}_k(\Omega)}$ on the restricted space $\Hk$ such that:
\begin{enumerate}
\item \textbf{Uniform norm equivalence.} There exist constants $c_{\mathcal{G}_k}, C_{\mathcal{G}_k} > 0$, depending only on $k$ (hence independent of $\Omega$), for which
\[
c_{\mathcal{G}_k}\,\normW{g}{W_2^{\tau}(\Omega)}
\;\le\;
\normW{g}{\mathcal{G}_k(\Omega)}
\;\le\;
C_{\mathcal{G}_k}\,\normW{g}{W_2^{\tau}(\Omega)}
\qquad \forall\, g \in \Hk .
\]

\item \textbf{Additive decomposition on finite partitions.}
For every $f \in \HkR$ and every finite family of pairwise disjoint open sets $\{\Omega_j\}_{j=1}^{P} \subset \R^{N}$ satisfying
$\overline{\bigcup_{j=1}^{P} \Omega_j} = \R^{N}$, we have
\[
\normW{f}{\HkR}^{2}
\;=\;
\sum_{j=1}^{P} \normW{\,f|_{\Omega_j}\,}{\mathcal{G}_k(\Omega_j)}^{2}.
\]
\end{enumerate}
\end{definition}

\noindent The local norm \(\normW{\cdot}{\mathcal{G}_k(\Omega)}\) need not coincide with the RKHS norms induced by the restricted kernels \(k|_{\Omega\times\Omega}\); for instance, \(H^{1}(\R)\) has the Matérn kernel \eqref{eq:matern:12} with $c_{N,\nu}=\tfrac{1}{2}$ as reproducing kernel on \(\R\) and the exact splitting
\[
\normW{f}{W_2^{1}(\R)}^{2}
\;=\;
\normW{\,f|_{(a,b)}\,}{W_2^{1}((a,b))}^{2}
\;+\;
\normW{\,f|_{\R\setminus(a,b)}\,}{W_2^{1}(\R\setminus(a,b))}^{2},
\]
whereas the reproducing kernel of \(H^{1}((a,b))\) is \eqref{eq:reproH1} and not the restriction of \eqref{eq:matern:12}, motivating the introduction of a new norm \(\normW{\cdot}{\mathcal{G}_k(\Omega)}\) in Definition~\ref{def:SObolevWIthLocalNOrm}.\\

\noindent
A broad subclass of Sobolev {kernels admitting a local norm decomposition} arises when \(\tau:=r\in\N\) with \(\tau>\tfrac{N}{2}\) and the spectral density is a polynomial in \(\|\omega\|_{2}^{2}\) with positive coefficients, i.e.,
\begin{equation}\label{eq:polySpec}
  \widehat{\Phi}(\omega)\;=\;\bigl(p(\|\omega\|_{2}^{2})\bigr)^{-1},
  \qquad
  p(t)=\sum_{m=0}^{\tau} a_{m}\,t^{m},\quad a_{m}>0\ (m=0,\dots,\tau).
\end{equation}
By Plancherel’s theorem  and
\(\widehat{\partial^{\alpha} f}(\omega)=(i\omega)^{\alpha}\,\widehat{f}(\omega)\) for all multi-indices \(\alpha\in\N_{0}^{N}\) with \(|\alpha|\le \tau\) (see \cite[Chapter~5]{W04}), one obtains
\begin{equation}\label{eq:localEnergy}
  \|f\|_{\HkR}^{2}
  \;=\;(2\pi)^{-N/2}\!\int_{\R^{N}} p(\|\omega\|_{2}^{2})\,|\widehat{f}(\omega)|^{2}\,\mathrm{d}\omega
  \;=\;(2\pi)^{-N/2}\!\sum_{m=0}^{\tau} a_{m}\!\!\sum_{|\alpha|=m}\!\frac{m!}{\alpha!}\,
  \|\partial^{\alpha} f\|_{L^{2}(\R^{N})}^{2},
\end{equation}
where \(\alpha!:=\alpha_{1}!\cdots \alpha_{N}!\) and the second equality uses the multinomial expansion
\[
  \bigl(\|\omega\|_{2}^{2}\bigr)^{m}
  \;=\;(\omega_{1}^{2}+\cdots+\omega_{N}^{2})^{m}
  \;=\;\sum_{|\alpha|=m}\frac{m!}{\alpha!}\,\omega^{2\alpha}.
\]
For any open set $\Omega\subset\R^{N}$ define the local weighted Sobolev norm (which plays the role of the $\mathcal{G}_{k}(\Omega)$ norm in Definition~\ref{def:SObolevWIthLocalNOrm}) by
\begin{equation}\label{eq:def-local-G}
  \|g\|_{W^{\tau}_{2,p}(\Omega)}^{2}
  \;:=\;(2\pi)^{-N/2}\sum_{m=0}^{\tau} a_{m}\!\!\sum_{|\alpha|=m}\!\frac{m!}{\alpha!}\,
      \|\partial^{\alpha} g\|_{L^{2}(\Omega)}^{2},
  \qquad g\in W^{\tau}_{2}(\Omega).
\end{equation}
For every finite family of pairwise disjoint open sets $\{\Omega_j\}_{j=1}^{P} \subset \R^{N}$ satisfying
$\overline{\bigcup_{j=1}^{P} \Omega_j} = \R^{N}$, we have  
\[
\normW{f}{\HkR}^{2}
\;=\;
\sum_{j=1}^{P} \normW{\,f|_{\Omega_j}\,}{W^{\tau}_{2,p}(\Omega_j)}^{2}.
\]
Moreover, \(W^{\tau}_{2,p}(\Omega)\simeq W^{\tau}_{2}(\Omega)\), and the norms are equivalent. Indeed,  
we have for all \(f\in W_{2}^{\tau}(\Omega)\),
\[
  (2\pi)^{-N/2}\,\min_{\substack{0\le m\le \tau\\ |\alpha|=m}}\Bigl\{a_{m}\,\frac{m!}{\alpha!}\Bigr\}\,\|f\|_{W_{2}^{\tau}(\Omega)}^{2}
  \;\le\; \|f\|_{W^{\tau}_{2,p}(\Omega)}^{2}
  \;\le\; (2\pi)^{-N/2}\,\max_{\substack{0\le m\le \tau\\ |\alpha|=m}}\Bigl\{a_{m}\,\frac{m!}{\alpha!}\Bigr\}\,\|f\|_{W_{2}^{\tau}(\Omega)}^{2}.
\]
In particular, choosing \(\widehat{\Phi}(\omega)=(1+\|\omega\|_{2}^{2})^{-\tau}\) with \(\tau\in\N\) corresponds to
\(p(t)=(1+t)^{\tau}\) with strictly positive binomial coefficients, so Matérn kernels with integer
\(\tau\) are Sobolev kernels admitting a local norm decomposition and 
 constitute prototypical members of this class.

\subsection{Bounds for functions in Sobolev spaces with scattered zeros}\label{sec:ScatterZero}
\noindent Generally, the error analysis for kernel interpolation with Sobolev kernels is based on bounds for functions in Sobolev spaces with scattered zeros, which is also crucial for the subsequent escaping the native space analysis; to formulate it, we introduce two geometric quantities that measure, respectively, how well a node set fills the domain and how well its points are separated.
\begin{definition}\label{def:fillSep}
For a finite set $X_n \subset \Omega \subset \mathbb{R}^N$, the  {fill distance} and {separation distance} are
\[
h_{X_n,\Omega} \;:=\; \sup_{x \in \Omega}\;\min_{y \in X_n}\,\|x-y\|_{2},
\qquad
q_{X_n} \;:=\; \tfrac{1}{2}\min_{\substack{x,y \in X_n\\ x \neq y}}\|x-y\|_{2}.
\]
Moreover, the sequence $\{X_n\}_{n\in\mathbb{N}}$ is  {quasi‐uniform} with ratio $\rho\ge1$ if
\[
\frac{h_{X_n,\Omega}}{q_{X_n}}\;\le\;\rho\qquad\text{for all }n\in\mathbb{N}.
\]
\end{definition}
\noindent Geometrically, $h_{X_n,\Omega}$ is the radius of the largest ball centered in $\Omega$ that does not meet $X_n$ (so $h_{X_n,\Omega}\!\to0$ means that $X_n$ increasingly fills $\Omega$), whereas $q_{X_n}$ equals half the minimal pairwise distance in $X_n$ and prevents clustering; by construction,
$
q_{X_n} 
\;\le\;
 h_{X_n,\Omega}.
$
Thus, for a quasi‐uniform sequence with mesh ratio $\rho$, the quotient $h_{X_n,\Omega}/q_{X_n}$ is uniformly bounded between $1$ and $\rho$. If $\Omega\subset\mathbb{R}^N$ is bounded and $\{X_n\}_{n\in\mathbb{N}}$ is {quasi‐uniform}, Proposition~14.1 in \cite{W04} yields the fill-distance estimate
\begin{equation}\label{eq:fillN}
h_{X_n,\Omega}\;\le\;C_{\Omega,N,\rho}\,n^{-1/N},
\end{equation}
where $C_{\Omega,N,\rho}$ depends only on $\Omega$, the ambient dimension $N$, and the mesh ratio $\rho$, and \eqref{eq:fillN} will be used repeatedly below.
\noindent To obtain convergence rates we assume an interior cone condition.
\begin{definition}[interior cone condition]\label{def:interior}
A set $\Omega\subset\mathbb{R}^N$ satisfies an interior cone condition if there exist an angle $\eta\in(0,\tfrac{\pi}{2})$ and a radius $R>0$ such that for every $x\in\Omega$ there is a unit vector $\zeta(x)\in\mathbb{R}^N$ with
\[
\bigl\{\,y\in\mathbb{R}^N:\ 0<\|y-x\|_{2}\le R\ \text{ and }\ \zeta(x)\cdot\tfrac{y-x}{\|y-x\|_{2}}\ge \cos\eta\,\bigr\}\ \subset\ \Omega.
\]
\end{definition}
\noindent In particular, every bounded Lipschitz domain $\Omega\subset\mathbb{R}^N$ satisfies an interior cone condition for some parameters $(\eta,R)$ that depend only on the Lipschitz constants of the local graphs describing $\partial\Omega$; see \cite[Lem.~1.5]{Krieg2023}.
\noindent With these notions in place we state the sampling inequality that underpins the convergence of Sobolev  kernel interpolation; it is a specialization of \cite[Thm.~11.32]{W04} to our notation.
\begin{theorem}\label{theo:EstSemiNormSobolev}
Let $\Omega \subset \mathbb{R}^N$ be a bounded set that satisfies an interior cone condition with angle $\eta$ and radius $R$, and let \(\tau>0\) satisfy \(\lfloor \tau \rfloor > m + \frac{N}{2}\) with $m \in \mathbb{N}_0$. If $X_n\subset\Omega$ has fill distance $h_{X_n,\Omega}\le h_1(\tau,\eta)\,R$, where
\begin{align}\label{eq:interiorH}
h_1(\tau,\eta) \;:=\; \frac{\sin\eta\,\sin\eta_0}{8\tau^2\,(1+\sin\eta)(1+\sin\eta_0)}\quad\text{with}\quad \eta_0 \;:=\; 2\arcsin\!\Bigl(\frac{\sin\eta}{4(1+\sin\eta)}\Bigr),
\end{align}
then every $u\in W_2^\tau(\Omega)$ with $u(x)=0$ for all $x\in X_n$ satisfies
\begin{align}\label{eq:resTheoScatterZero}
|u|_{W_2^{m}(\Omega)} \;\le\; C_{\text{\tiny{SZ}}}\,h_{X_n,\Omega}^{\,\tau-m}\,|u|_{W_2^{\tau}(\Omega)} \quad\text{and}\quad \|u\|_{L^{\infty}(\Omega)} \;\le\; C_{\text{\tiny{SZ}}}\,h_{X_n,\Omega}^{\,\tau-N/2}\,|u|_{W_2^{\tau}(\Omega)}
,
\end{align}
where $C_{\text{\tiny{SZ}}}>0$ depends only on $N,\tau,\eta$ and is independent of $u,h_{X_n,\Omega},R$, and $\Omega$.
\end{theorem}
\begin{proof}
See \cite[Thm.~11.32]{W04}.
\end{proof}
\noindent
Assume now that $\Omega \subset \mathbb{R}^{N}$ is a bounded Lipschitz domain and that the hypotheses of Theorem~\ref{theo:EstSemiNormSobolev} are satisfied for the corresponding cone parameters $(\eta,R)$.  Let   $s_f^n$ denote the kernel interpolant of target function $f\in \Hk$ at the nodes $X_n$ and set $u:=f-s_f^n$, we obtain
\[
\lvert f-s_f^n\rvert_{W_2^m(\Omega)}
\;\le\;
C_{\text{\tiny{SZ}}}'\, h_{X_n,\Omega}^{\,\tau-m}\,\|f\|_{\Hk} \quad\text{and}\quad\lVert f-s_f^n\rVert_{L^{\infty}(\Omega)}
\;\le\;
C_{\text{\tiny{SZ}}}'\, h_{X_n,\Omega}^{\,\tau-N/2}\,\|f\|_{\Hk}. 
\]
Here $C_{\text{\tiny{SZ}}}'>0$ depends only on $N,\tau,\eta$ and on the norm-equivalence constants between $W_2^\tau(\Omega)$ and $\Hk$ (in particular, $C_{\text{\tiny{SZ}}}'$ is independent of $f$ and of the specific choice of $X_n$). Note that $u \in W_2^{\tau}(\Omega)\simeq\Hk$, since $s_f^n \in \Hk$ by construction. The estimate yields the standard convergence
rates for Sobolev kernel interpolation under the native-space assumption
$f \in \Hk$. In the next section we relax this assumption and study convergence for target functions $f\notin\Hk$, for which precise algebraic rates are, in general, not available.

\section{Escaping the native space}\label{sec3}
In this section, we provide a general framework to analyze approximation of target functions lying outside the native RKHS and to identify the norms in which convergence is attained. The principal device for this analysis
is the interpolation operator, which, in standard form for
$
  X_n=\{x_1,\dots,x_n\}\subset\Omega\subset\R^{N}
$
 pairwise distinct centers and $k$ s.p.d.\! with native space $\Hk$, is given by
\eqref{eq:Pi_Hk}. To extend interpolation to target functions $f\notin\mathcal{H}_k(\Omega)$,  we introduce  precise function‐space criteria that (i) specify the Banach space of functions in which the target function lies, and (ii) identify the Banach space of functions  whose norm governs the convergence of the interpolant.
\begin{definition}[$k$-admissible pair of Banach spaces]\label{def:admis-closure}
Let $\Omega\subset\R^{N}$ be a domain and let $k:\Omega\times\Omega\to\R$ be s.p.d.\! with native space $\Hk$.
Two Banach spaces of functions on $\Omega$, denoted $A(\Omega)$ and $B(\Omega)$, form a {$k$-admissible pair on $\Omega$} if we have:
\begin{enumerate}
  \item[(a)] \textbf{Bounded point evaluations in $A(\Omega)$:} there exists $c_A>0$ with
  \[
    |f(x)|\le c_A\|f\|_{A(\Omega)}\qquad\forall\,f\in A(\Omega),~\forall\,x\in\Omega.
  \]
  \item[(b)] \textbf{Continuous embedding:} $\Hk\hookrightarrow A(\Omega)$, i.e., there exists $C_A>0$ with
  \[
    \|f\|_{A(\Omega)}\le C_A\|f\|_{\Hk}\qquad\forall\,f\in \Hk.
  \]
    \item[(c)] \textbf{Density:} $\overline{ \Hk }^{\,\|\cdot\|_{A(\Omega)}}=A(\Omega)$.
  \item[(d)] \textbf{Target-norm control:} $A(\Omega)\hookrightarrow B(\Omega)$, i.e., there exists $C_B>0$ with
  \[
    \|f\|_{B(\Omega)}\le C_B\|f\|_{A(\Omega)}\qquad\forall\,f\in A(\Omega).
  \]
\end{enumerate}
\end{definition}

\noindent Given a $k$-admissible pair $(A(\Omega),B(\Omega))$, we define the extended interpolation operator by
\begin{equation}\label{eq:Pi_C_B}
  \Pi^n_{A,B}:A(\Omega)\to B(\Omega),\qquad
  \Pi^n_{A,B}f=\sum_{i=1}^{n} f(x_i)\,l_{x_i}(\,\cdot\,),
\end{equation}
where the  functions \(\{l_{x_i}\}_{i=1}^n\subset A(\Omega)\subset B(\Omega)\).  Since
\begin{align*}
\bigl\lVert \Pi^n_{A,B} f\bigr\rVert_{B(\Omega)}
&= \left\Vert \sum_{i=1}^n f(x_i)\,l_{x_i}\right\Vert_{B(\Omega)}
\;\le\;
\max_{i=1,...,n}|f(x_i)| \;\sum_{i=1}^n \left\Vert   l_{x_i}\right\Vert_{B(\Omega)} \;\le\;
c_A\lVert f\rVert_{A(\Omega)} \; \sum_{i=1}^n \left\Vert   l_{x_i}\right\Vert_{B(\Omega)},
\end{align*}
it follows that \(\Pi^n_{A,B}\in\mathcal{L}\bigl(A(\Omega),B(\Omega)\bigr)\), with bounded operator norm
\begin{equation}
  \label{eq:operatorNormEstimate}
  \bigl\lVert \Pi^n_{A,B}\bigr\rVert_{\mathcal{L}(A(\Omega),B(\Omega))}
  \;\le\; c_A
    \sum_{i=1}^n \left\Vert   l_{x_i}\right\Vert_{B(\Omega)}.
\end{equation}
\noindent The following theorem -- one of the main contributions of this work -- demonstrates that the growth rate of the operator norm $\bigl\lVert \Pi^n_{A,B}\bigr\rVert_{\mathcal{L}(A(\Omega),B(\Omega))}$ is the decisive factor governing the convergence of kernel interpolation in the $B(\Omega)$-norm for target functions in $A(\Omega)$.
\begin{theorem}\label{theo:conncentionLebeEscaping}
Let $\Omega\subset\R^{N}$ be a domain and let $k:\Omega\times\Omega\to\R$ be a continuous, s.p.d.\! kernel with RKHS $\Hk$  and let $(A(\Omega),B(\Omega))$ be a $k$-admissible pair. Suppose  $(X_n)_{n\in\N}$ is  a nested sequence of finite, pairwise distinct point sets $X_n\subset\Omega$ with dense union,
\[
  X_1\subset X_2\subset\cdots,\qquad \overline{\bigcup_{n\in\N} X_n}=\overline{\Omega}.
\]
Then
\[
  \lim_{n\to\infty}\bigl\|f-\Pi_{A,B}^{n}f\bigr\|_{B(\Omega)}=0\quad\text{for all }f\in A(\Omega)
\]
if and only if the sequence of operator norms $\bigl(\|\Pi^n_{A,B}\|_{\mathcal{L}(A(\Omega),B(\Omega))}\bigr)_{n\in\N}$ is uniformly bounded in $n$.
\end{theorem}
\begin{proof}
($\Rightarrow$) This is a direct consequence of the Uniform Boundedness Principle (see, e.g., \cite[Theorem 14.1]{conway2019course}).

\medskip
\noindent
($\Leftarrow$) If $\|\Pi^n_{A,B}\|\le\Lambda<\infty$ for all $n$, fix $f\in A(\Omega)$ and $\varepsilon>0$. By density, choose $g\in \Hk$ with $\|f-g\|_{A(\Omega)}\le \varepsilon\big/ \bigl(2\,C_B(1+\Lambda)\bigr)$. By Theorem~\ref{theo:convGeneral} (all its assumptions are satisfied in the present setting), we have
\[
  \|g - \Pi_k^n g\|_{\Hk} \longrightarrow 0
  \quad \text{as } n \to \infty.
\]
Thus there exists $N \in \N$ such that, for all $n \ge N$,
\begin{align*}
  \|g-\Pi^n_{A,B}g\|_{A(\Omega)}  =\left\| g-\sum_{i=1}^{n} g(x_i)\,l_{x_i} \right\|_{A(\Omega)}\le C_A\,\|g-\Pi^n_k g\|_{\Hk}\le \varepsilon\big/\bigl(2\,C_B\bigr).
\end{align*}
Therefore, combing these two inequalities gives, for $n\ge N$,
\begin{align*}
\|f-\Pi^n_{A,B}f\|_{B(\Omega)}
&\le \|f- g\|_{B(\Omega)}+\| g-\Pi^n_{A,B} g\|_{B(\Omega)}+\|\Pi^n_{A,B}( g-f)\|_{B(\Omega)} \\
&\le C_B\|f- g\|_{A(\Omega)}+C_B\| g-\Pi^n_{A,B} g\|_{A(\Omega)}+\Lambda\,C_B\| g-f\|_{A(\Omega)} \\
&\le \frac{\varepsilon}{2}+\frac{\varepsilon}{2}=\varepsilon,
\end{align*}
which proves $\|f-\Pi^n_{A,B}f\|_{B(\Omega)}\to 0$ for $n \to \infty$.
\end{proof}
\noindent In what follows, we focus on Sobolev kernels and on $k$-admissible pairs with merely continuous target functions on $\overline{\Omega}$ and   $L^2(\Omega)$-norm or the supremum norm for the error metric. We analyze the asymptotic behavior of the operator norms of $\Pi_{A,B}^{n}$ to derive rigorous ``escaping the native space'' results. Nevertheless, Theorem~\ref{theo:conncentionLebeEscaping} supplies a general abstract principle that requires only an s.p.d.\! kernel and a $k$-admissible pair, so that alternative choices of $(A(\Omega),B(\Omega))$ are likewise conceivable whenever the conditions of Definition~\ref{def:admis-closure} are met, a point to which we return in the concluding section.

\section{Approximating continuous functions by Sobolev  kernel interpolation in the \(L^2\)-norm}\label{sec4}
In this section we consider
\begin{align}\label{eq:kAd1}
     A(\Omega)=C_{\mathrm{ex}}(\Omega), 
     \qquad 
     B(\Omega)=L^2(\Omega),
\end{align}
on a bounded Lipschitz domain \(\Omega\subset\R^N\), and a Sobolev kernel \(k\) of order \(\tau\) with \(\lfloor \tau \rfloor>\tfrac{N}{2}\). Here,
 we work with the space
\[
C_{\mathrm{ex}}(\Omega)
:=\Bigl\{\, f\in C(\Omega)\;:\;\exists\,\tilde f\in C(\overline{\Omega})\text{ with }\tilde f|_{\Omega}=f \Bigr\},
\quad
\normW{f}{C_{\mathrm{ex}}(\Omega)}
:= \sup_{x\in\Omega}|f(x)|
= \sup_{x\in\overline{\Omega}}|\tilde f(x)|,
\]
i.e., the continuous functions on \(\Omega\) that admit a continuous extension to the compact closure \(\overline{\Omega}\).
The restriction map \(R:C(\overline{\Omega})\to C_{\mathrm{ex}}(\Omega)\) and the extension map \(E:C_{\mathrm{ex}}(\Omega)\to C(\overline{\Omega})\) provide an isometric isomorphism \(C_{\mathrm{ex}}(\Omega)\cong C(\overline{\Omega})\).
Working with \(C_{\mathrm{ex}}(\Omega)\) (equivalently \(C(\overline{\Omega})\)) is crucial for applying the Sobolev embedding theorems and for the density of the Sobolev space in \(C_{\mathrm{ex}}(\Omega)\); note that these features are not available in general for \(C_b(\Omega)\), whose elements need not have a continuous extension on \(\overline{\Omega}\). That the choice \eqref{eq:kAd1} forms a $k$-admissible pair is the content of the following proposition.

\begin{proposition}\label{pro:adCL}
Let \(\Omega \subset \mathbb{R}^N\) be a bounded domain with Lipschitz boundary, and let \(k\) be a Sobolev  kernel of order $\tau$ with  \(\lfloor \tau \rfloor>\tfrac{N}{2}\).  Then \(\bigl(A(\Omega),B(\Omega)\bigr):=\bigl(C_{\mathrm{ex}}(\Omega),L^2(\Omega)\bigr)\) is a $k$-admissible   pair.
\end{proposition}

\begin{proof}
We verify the four conditions of Definition~\ref{def:admis-closure}:
\begin{itemize}
  \item[a)] Clearly, 
    \[
      |f(x)| \;\le\; \sup_{y\in\Omega}|f(y)| \;=\;\|f\|_{C_{\mathrm{ex}}(\Omega)},
      \quad
      \forall\,x\in\Omega,\;f\in C_{\mathrm{ex}}(\Omega).
    \]
    Hence \(c_{A}=1\).
    \item[b)] 
    First, we prove the continuity of the embedding \(H^{\tau}(\Omega)\hookrightarrow H^{\lfloor\tau\rfloor}(\Omega)\). This is an immediate consequence of the corresponding embedding on the whole space \(\R^{N}\), which follows from the Fourier characterization: for every \(f\in H^{\tau}(\R^{N})\),
\begin{align*}
 \|f\|_{H^{\lfloor\tau\rfloor}(\R^{N})}^{2}
 &= (2\pi)^{-N/2}\int_{\R^{N}} |\widehat{f}(\omega)|^{2}\,(1+\|\omega\|_{2}^{2})^{\lfloor\tau\rfloor}\,\mathrm{d}\omega \\
 &\le (2\pi)^{-N/2}\int_{\R^{N}} |\widehat{f}(\omega)|^{2}\,(1+\|\omega\|_{2}^{2})^{\tau}\,\mathrm{d}\omega
 = \|f\|_{H^{\tau}(\R^{N})}^{2},
\end{align*}
since \(\lfloor\tau\rfloor\le \tau\). Note that, because \(\tau\ge \lfloor\tau\rfloor > N/2\), the space \(H^{\tau}(\Omega)\) and \(H^{\lfloor\tau\rfloor}(\Omega)\) are restricted RKHS according to \eqref{eq:restRKHS1}-\eqref{eq:restRKHS2} associated to the corresponding Matérn kernels. This directly yields  \(H^{\tau}(\Omega)\hookrightarrow H^{\lfloor\tau\rfloor}(\Omega)\). Combining the embedding \(H^{\tau}(\Omega)\hookrightarrow H^{\lfloor\tau\rfloor}(\Omega)\) with the identifications and norm equivalences stated in \eqref{eq:NormEquiLocal} (namely, for \(\Hk\) and \(W_{2}^{\tau}(\Omega)\), for \(H^{\tau}(\Omega)\) and \(W_{2}^{\tau}(\Omega)\), and for \(H^{\lfloor\tau\rfloor}(\Omega)\) and \(W_{2}^{\lfloor\tau\rfloor}(\Omega)\)), we conclude the continuous embedding
\[
\Hk \hookrightarrow W_{2}^{\lfloor\tau\rfloor}(\Omega).
\]
Furthermore, if \(\lfloor\tau\rfloor>\tfrac{N}{2}\) and \(\Omega\) is a Lipschitz domain, there exists a continuous extension and embedding operator
$
 E: W_{2}^{\lfloor\tau\rfloor}(\Omega) \to C (\overline\Omega),
$
which maps each \(f \in W_{2}^{\lfloor\tau\rfloor}(\Omega)\) uniquely to \(  E f \in C (\overline \Omega)\) and satisfies
\[
\|  E f\|_{C (\overline\Omega)} \le C_{  E}\, \|f\|_{W_{2}^{\lfloor\tau\rfloor}(\Omega)}
\quad \text{(see \cite[Chapter 4]{adams2003}).}
\]
Combining the continuous embedding \(\tilde \iota: \Hk \hookrightarrow W_{2}^{\lfloor\tau\rfloor}(\Omega)\) with \( E\) and the restriction map \(R:C(\overline{\Omega})\to C_{\mathrm{ex}}(\Omega)\) yields the continuous embedding operator
\[
\iota :=  R \circ  E \circ \tilde\iota : \Hk \longrightarrow C_{\mathrm{ex}}(\Omega).
\]
 In particular, there is a constant $C_{A }>0$  such that
\[
\|\iota f\|_{C_{\mathrm{ex}}(\Omega)} \le C_{A }  \, \|f\|_{\Hk} \qquad \text{for all } f\in \Hk.
\]

 \item[c)] 
By the same reasoning as in the previous part of the proof, and since
$\lceil \tau \rceil \ge \lfloor \tau \rfloor > \tfrac{N}{2}$, we also conclude that
$ W_2^{\lceil \tau \rceil}(\Omega) \subset W_2^{\tau}(\Omega) \simeq \Hk.$
Additionally, as a consequence of the Stone–Weierstrass theorem, for all
$f \in C(\overline{\Omega})$ and for all $\varepsilon > 0$, there exists a polynomial
$p \in \mathcal{P}(\overline{\Omega})$ with
\begin{align*}
  \| f - p \|_{C(\overline{\Omega})}
  = \| f|_{\Omega} - p|_{\Omega} \|_{C_{\mathrm{ex}}(\Omega)}
  < \varepsilon .
\end{align*}
Furthermore, every polynomial is smooth, hence $p|_{\Omega} \in W_2^{\lceil \tau \rceil}(\Omega)$,
and since $W_2^{\lceil \tau \rceil}(\Omega) \subset \Hk \subset C_{\mathrm{ex}}(\Omega)$, it follows that
\[
  \overline{\Hk}^{\,\|\cdot\|_{C_{\mathrm{ex}}(\Omega)}} \;=\; C_{\mathrm{ex}}(\Omega).
\]

  \item[d)] Finally, since \(\Omega\) is bounded,
    \[
      \|f\|_{L^2(\Omega)}
      \;\le\;\sqrt{|\Omega|}\,\|f\|_{C_{\mathrm{ex}}(\Omega)},
      \quad
      \forall\,f\in C_{\mathrm{ex}}(\Omega),
    \]
    so one may take \(C_{B}=\sqrt{|\Omega|}\).
\end{itemize}
\end{proof}
\noindent 
\noindent We now establish the principal result of this section, which asserts that the interpolation operator remains uniformly bounded in operator norm from \(C_{\mathrm{ex}}(\Omega)\) to \(L^2(\Omega)\) under the assumptions of quasi-uniform distribution of the interpolation nodes. The proof follows the strategy used in the proof of Theorem~1 of \cite{marchi2010}, but in a different norm.
\begin{theorem}\label{th:CLuniOp}
Let \(k\) be a Sobolev kernel of order \(\tau\) with \(\lfloor \tau \rfloor > N/2\). Let \(\Omega \subset \R^{N}\) be a bounded Lipschitz domain satisfying the interior cone condition for the angle \(\eta \in (0,\pi/2)\) and radius \(R>0\). Consider node sets \(X_n \subset \Omega\) with fill distance \(h_{X_n,\Omega}\) and separation distance \(q_{X_n}\) such that
\[
h_{X_n,\Omega} \le h_1(\tau,\eta)\,R,\qquad q_{X_n}<1,\qquad h_{X_n,\Omega} \le \rho\, q_{X_n}
\]
for some \(\rho \ge 1\), where \(h_1(\tau,\eta)\) is the constant from~\eqref{eq:interiorH}. Then the kernel interpolation operator \(\Pi^{n}_{C,L^2}: C_{\mathrm{ex}}(\Omega) \to L^{2}(\Omega)\) associated with \(k\) and centers \(X_n\) satisfies
\[
\bigl\|\Pi^{n}_{C,L^2}\bigr\|_{\mathcal{L}\!\left(C_{\mathrm{ex}}(\Omega),\,L^{2}(\Omega)\right)} \le C_{\Pi_{C,L^2}},
\]
where \(C_{\Pi_{C,L^2}}>0\) depends only on \(k\), on the geometric parameters of \(\Omega\), and on \(\rho\), and is independent of \(n\) and of the particular choice of \(X_n\).
\end{theorem}
\begin{proof}
Let \(\phi\in C_c^\infty(\mathbb{R}^N)\) be a standard bump function with 
\begin{align}\label{eq:standartBump}
\operatorname{supp}\phi\subset B(0,1), 
\quad \phi(0)=1,
\quad \|\phi\|_{L^2(\mathbb{R}^N)}=1,
\end{align}
then  $\phi (\frac{\cdot-x_i}{q_{X_n}} )$ has support in $B(x_i,q_{X_n})$. 
 For any \(f\in C_{\mathrm{ex}}(\Omega)\), we have
\begin{align}
\bigl\|\Pi^n_{C,L^2}f\bigr\|_{L^2(\Omega)}
=& \Bigl\|\sum_{i=1}^n f(x_i)\, l_{x_i}\Bigr\|_{L^2(\Omega)} \notag \\
\le &
\Bigl\|\sum_{i=1}^n f(x_i)\bigl(l_{x_i}-\phi(\tfrac{\cdot-x_i}{q_{X_n}})\bigr)\Bigr\|_{L^2(\Omega)}
+
\Bigl\|\sum_{i=1}^n f(x_i)\,\phi(\tfrac{\cdot-x_i}{q_{X_n}})\Bigr\|_{L^2(\Omega)}.\label{eq:chain000}
\end{align}
We bound the two latter terms separately.\\

\noindent\textbf{(i) The remainder term.}
Set
\[
p_f \;:=\; \sum_{i=1}^n f(x_i)\,\phi\!\Bigl(\tfrac{\cdot-x_i}{q_{X_n}}\Bigr).
\]
Then, we obtain (the arguments for each step are given below)
\begin{align}
\Bigl\Vert \sum_{i=1}^n f(x_i)\!\Bigl( l_{x_i} - \phi\!\Bigl(\tfrac{\cdot-x_i}{q_{X_n}}\Bigr)\Bigr) \Bigr\Vert_{L^2(\Omega)}
&\le C_{\text{\tiny{SZ}}}\, h_{X_n,\Omega}^{\tau} 
\Bigl\Vert \sum_{i=1}^n f(x_i) \Bigl( l_{x_i} - \phi\!\Bigl(\tfrac{\cdot-x_i}{q_{X_n}}\Bigr)\Bigr) \Bigr\Vert_{W^\tau_2(\Omega)} \label{eq:chain2}\\
&\le \frac{C_{\text{\tiny{SZ}}}}{c_{\tau,\Omega}}\, h_{X_n,\Omega}^{\tau} 
\Bigl\Vert \sum_{i=1}^n f(x_i)\!\Bigl( l_{x_i} - \phi\!\Bigl(\tfrac{\cdot-x_i}{q_{X_n}}\Bigr)\Bigr) \Bigr\Vert_{\Hk} \label{eq:chain3}\\
&= \frac{C_{\text{\tiny{SZ}}}}{c_{\tau,\Omega}}\, h_{X_n,\Omega}^{\tau}\,
\bigl\Vert \Pi^n p_f - p_f \bigr\Vert_{\Hk} \label{eq:chain4}\\
&\le \frac{C_{\text{\tiny{SZ}}}}{c_{\tau,\Omega}}\, h_{X_n,\Omega}^{\tau}\,
\bigl\Vert \Pi^n - I \bigr\Vert_{\mathcal{L}(\Hk)}\, \Vert p_f \Vert_{\Hk} \label{eq:chain5}\\
&\le \frac{C_{\text{\tiny{SZ}}}}{c_{\tau,\Omega}}\, h_{X_n,\Omega}^{\tau}\, \Vert p_f \Vert_{\Hk} \label{eq:chain6}\\
&\le \frac{C_{\text{\tiny{SZ}}}}{c_{\tau,\Omega}}\, h_{X_n,\Omega}^{\tau}\, \Vert p_f \Vert_{\HkR} \label{eq:chain7}\\
&\le \frac{C_{\text{\tiny{SZ}}} C_{\tau,\R^N}}{c_{\tau,\Omega}}\, h_{X_n,\Omega}^{\tau}\, \Vert p_f \Vert_{W_2^{\tau}(\R^N)}. \label{eq:chain8}
\end{align}
Here,  \eqref{eq:chain2} follows from is the scattered-zeros estimate \eqref{eq:resTheoScatterZero} for $m:=0$ (applicable since $\phi \in W_2^{\tau}(\Omega)$ as smooth function); \eqref{eq:chain3} uses the norm equivalence \eqref{eq:NormEquiLocal}; \eqref{eq:chain4} follows from \(p_f(x_i)=f(x_i)\) so \(\Pi^n p_f=\sum_i f(x_i)l_{x_i}\); \eqref{eq:chain5} is submultiplicativity; \eqref{eq:chain6} uses the identity for the norm of the orthogonal projection given in \eqref{eq:operotOne}; \eqref{eq:chain7} follows from the minimal-norm extension property \eqref{eq:restRKHS2}; and \eqref{eq:chain8} is the norm equivalence between $\HkR$ and $W_2^{\tau}(\R^N)$ for Sobolev kernels. 
Let $\tau=\lfloor\tau\rfloor+s$ with $s\in[0,1)$. We have
\begin{equation}\label{eq:chain10}
  \|p_f\|_{W_2^{\tau}(\R^N)}^{2}
  \;:=\; \sum_{m=0}^{\lfloor\tau\rfloor} \left|p_f\right|_{W_2^m(\R^N)}^{2}
  \;+\; \left|p_f\right|_{W_2^{\lfloor\tau\rfloor+s}(\R^N)}^{2}.
\end{equation}
Since the $\phi\!\bigl((\cdot-x_i)/q_{X_n}\bigr)$ are translate-dilates with pairwise disjoint supports and
$\partial^\alpha\phi\!\bigl((\cdot-x_i)/q_{X_n}\bigr)=q_{X_n}^{-m}(\partial^\alpha\phi)\!\bigl((\cdot-x_i)/q_{X_n}\bigr)$ for $|\alpha|=m$,
a change of variables in the $L^2$-norm yields
\begin{align}
\sum_{m=0}^{\lfloor\tau\rfloor} |p_f|_{W_2^m(\R^N)}^{2}
&= \sum_{i=1}^n |f(x_i)|^2 \sum_{m=0}^{\lfloor\tau\rfloor}\ \sum_{|\alpha|=m}
\bigl\| \partial^\alpha \phi\!\bigl(\tfrac{\cdot-x_i}{q_{X_n}}\bigr)\bigr\|_{L^2(\R^N)}^2 \notag\\
&= \sum_{i=1}^n |f(x_i)|^2 \sum_{m=0}^{\lfloor\tau\rfloor}\ \sum_{|\alpha|=m}
q_{X_n}^{\,N-2m}\,\|\partial^\alpha \phi\|_{L^2(\R^N)}^2 \label{eq:chain11}\\
&\le n\,\|f\|_{C_{\mathrm{ex}}(\Omega)}^2\, q_{X_n}^{\,N-2 \tau }\,
\sum_{m=0}^{\lfloor\tau\rfloor} |\phi|_{W_2^m(\R^N)}^{2}, \label{eq:chain12}
\end{align}
where the first equality comes from the property of the pairwise disjoint supports and  in \eqref{eq:chain12} we used $0<q_{X_n}\le 1$ so that $q_{X_n}^{N-2m}\le q_{X_n}^{N-2 \tau }$ for $m\le \lfloor\tau\rfloor$. 
For the Slobodeckij seminorm, which is only relevant when $\tau \notin \mathbb{N}_0$, we proceed similarly, but first record the following preliminary estimate valid for all multi-indices $\alpha$ with $|\alpha|=\lfloor\tau\rfloor$:
\begin{align}
\bigl|\partial^{\alpha}p_f(x)-\partial^{\alpha}p_f(y)\bigr|^2
&= \left| \sum_{i=1}^n f(x_i)\!\left(\partial^{\alpha}\phi\!\Bigl(\tfrac{x-x_i}{q_{X_n}}\Bigr)
 - \partial^{\alpha}\phi\!\Bigl(\tfrac{y-x_i}{q_{X_n}}\Bigr) \right) \right|^2 \notag \\
&\le 2\,\|f\|_{C_{\mathrm{ex}}(\Omega)}^2
 \sum_{i=1}^n \left| \partial^{\alpha}\phi\!\Bigl(\tfrac{x-x_i}{q_{X_n}}\Bigr)
 - \partial^{\alpha}\phi\!\Bigl(\tfrac{y-x_i}{q_{X_n}}\Bigr) \right|^2. \label{eq:chain13}
\end{align}
We prove this in two cases:\\
\emph{Case 1:} $x\in B(x_s,q_{X_n})$ and $y\in B(x_j,q_{X_n})$ with $s\neq j$. Then only the indices $s$ and $j$ contribute, and
\begin{align*}
\bigl|\partial^{\alpha}p_f(x)-\partial^{\alpha}p_f(y)\bigr|^2
&= \Bigl|\, f(x_s)\!\left(\partial^{\alpha}\phi\!\Bigl(\tfrac{x-x_s}{q_{X_n}}\Bigr)
- \partial^{\alpha}\phi\!\Bigl(\tfrac{y-x_s}{q_{X_n}}\Bigr)\right)
\\[-2pt]&\qquad\quad
+\, f(x_j)\!\left(\partial^{\alpha}\phi\!\Bigl(\tfrac{x-x_j}{q_{X_n}}\Bigr)
- \partial^{\alpha}\phi\!\Bigl(\tfrac{y-x_j}{q_{X_n}}\Bigr)\right) \Bigr|^{2} \\
&\le \|f\|_{C_{\mathrm{ex}}(\Omega)}^2
\Bigl(\, \bigl|\partial^{\alpha}\phi\!\Bigl(\tfrac{x-x_s}{q_{X_n}}\Bigr)
- \partial^{\alpha}\phi\!\Bigl(\tfrac{y-x_s}{q_{X_n}}\Bigr)\bigr|
\\[-2pt]&\qquad\qquad\qquad\qquad
+ \bigl|\partial^{\alpha}\phi\!\Bigl(\tfrac{x-x_j}{q_{X_n}}\Bigr)
- \partial^{\alpha}\phi\!\Bigl(\tfrac{y-x_j}{q_{X_n}}\Bigr)\bigr| \Bigr)^{2} \\
&\le 2\,\|f\|_{C_{\mathrm{ex}}(\Omega)}^2
\Bigl(\bigl|\partial^{\alpha}\phi\!\Bigl(\tfrac{x-x_s}{q_{X_n}}\Bigr)
- \partial^{\alpha}\phi\!\Bigl(\tfrac{y-x_s}{q_{X_n}}\Bigr)\bigr|^{2}
\\[-2pt]&\qquad\qquad\qquad\qquad
+ \bigl|\partial^{\alpha}\phi\!\Bigl(\tfrac{x-x_j}{q_{X_n}}\Bigr)
- \partial^{\alpha}\phi\!\Bigl(\tfrac{y-x_j}{q_{X_n}}\Bigr)\bigr|^{2}\Bigr) \\
&= 2\,\|f\|_{C_{\mathrm{ex}}(\Omega)}^2
 \sum_{i=1}^n \left| \partial^{\alpha}\phi\!\Bigl(\tfrac{x-x_i}{q_{X_n}}\Bigr)
 - \partial^{\alpha}\phi\!\Bigl(\tfrac{y-x_i}{q_{X_n}}\Bigr) \right|^2 ,
\end{align*}
where, in the last step, we added zero terms corresponding to the remaining indices.\\
\emph{Case 2:} $x,y\in B(x_s,q_{X_n})$. Then only the index $s$ contributes, and
\begin{align*}
\bigl|\partial^{\alpha}p_f(x)-\partial^{\alpha}p_f(y)\bigr|^2
&= \bigl| f(x_s)\bigr|^{2}\!
\left(\partial^{\alpha}\phi\!\Bigl(\tfrac{x-x_s}{q_{X_n}}\Bigr)
- \partial^{\alpha}\phi\!\Bigl(\tfrac{y-x_s}{q_{X_n}}\Bigr)\right)^{2} \\
&\le 2\,\|f\|_{C_{\mathrm{ex}}(\Omega)}^2
 \sum_{i=1}^n \left| \partial^{\alpha}\phi\!\Bigl(\tfrac{x-x_i}{q_{X_n}}\Bigr)
 - \partial^{\alpha}\phi\!\Bigl(\tfrac{y-x_i}{q_{X_n}}\Bigr) \right|^2 .
\end{align*}
In particular, the final upper bound in both cases is independent of the specific choice of $s$ and $j$. Using the seminorm definition and \eqref{eq:chain13}, we obtain
\begin{align}
    \left|p_f\right|_{W_2^{\lfloor\tau\rfloor+s}(\R^N)}^{2} =& \sum_{|\alpha|=\lfloor\tau\rfloor}\int_{\R^N}\!\int_{\R^N}
    \frac{\bigl|\partial^{\alpha}p_f(x)-\partial^{\alpha}p_f(y)\bigr|^{2}}{\|x-y\|_{2}^{\,N+2s}}\,
    \mathrm{d}x\,\mathrm{d}y, \notag \\
    \leq & 2\,\|f\|_{C_{\mathrm{ex}}(\Omega)}^2\sum_{|\alpha|=\lfloor\tau\rfloor}\sum_{i=1}^n\int_{\R^N}\!\int_{\R^N}
    \frac{\left| \partial^{\alpha}\phi\!\Bigl(\tfrac{x-x_i}{q_{X_n}}\Bigr)
 - \partial^{\alpha}\phi\!\Bigl(\tfrac{y-x_i}{q_{X_n}}\Bigr) \right|^2}{\|x-y\|_{2}^{\,N+2s}}\,
    \mathrm{d}x\,\mathrm{d}y, \notag\\
     =& 2\,\|f\|_{C_{\mathrm{ex}}(\Omega)}^2 q_{X_n}^{-2\lfloor\tau\rfloor}\sum_{|\alpha|=\lfloor\tau\rfloor}\sum_{i=1}^n\int_{\R^N}\!\int_{\R^N}
    \frac{\left| \left(\partial^{\alpha}\phi\right)\!\Bigl(\tfrac{x-x_i}{q_{X_n}}\Bigr)
 - \left(\partial^{\alpha}\phi\right)\!\Bigl(\tfrac{y-x_i}{q_{X_n}}\Bigr) \right|^2}{\|x-y\|_{2}^{\,N+2s}}\,
    \mathrm{d}x\,\mathrm{d}y,\label{eq:chain14} \\
     =& 2\,\|f\|_{C_{\mathrm{ex}}(\Omega)}^2 q_{X_n}^{-2\lfloor\tau\rfloor}\sum_{|\alpha|=\lfloor\tau\rfloor}\sum_{i=1}^n\int_{\R^N}\!\int_{\R^N}
    \frac{\left| \left(\partial^{\alpha}\phi\right)\! (u )
 - \left(\partial^{\alpha}\phi\right)\! (v ) \right|^2}{q_{X_n}^{N+2s}\|u-v\|_{2}^{\,N+2s}}\, q_{X_n}^{2N}
    \mathrm{d}u\,\mathrm{d}v,\label{eq:chain15}\\
    =& 2\,\|f\|_{C_{\mathrm{ex}}(\Omega)}^2 q_{X_n}^{N-2\lfloor\tau\rfloor-2s} n \left|\phi\right|_{W_2^{\lfloor\tau\rfloor+s}(\R^N)}^{2}\notag \\
      =& 2\,\|f\|_{C_{\mathrm{ex}}(\Omega)}^2 q_{X_n}^{N-2\tau} n \left|\phi\right|_{W_2^{\lfloor\tau\rfloor+s}(\R^N)}^{2}.\label{eq:chain16}
\end{align}
In \eqref{eq:chain14} we used the chain rule
and for \eqref{eq:chain15} the substitution $u=\tfrac{x-x_i}{q_{X_n}}$ and $v=\tfrac{y-x_i}{q_{X_n}}$,
so that $\mathrm{d}x\,\mathrm{d}y=q_{X_n}^{2N}\,\mathrm{d}u\,\mathrm{d}v$ and
$\|x-y\|_{2}^{\,N+2s}=q_{X_n}^{N+2s}\,\|u-v\|_{2}^{\,N+2s}$.
Combining \eqref{eq:chain10} with \eqref{eq:chain12} and \eqref{eq:chain16} yields
\begin{align*}
    \|p_f\|_{W_2^{\tau}(\R^N)} \leq \sqrt{2}\,\|f\|_{C_{\mathrm{ex}}(\Omega)} q_{X_n}^{N/2-\tau} \sqrt{n} \|\phi\|_{W_2^{\tau}(\R^N)},
\end{align*}
which, together with \eqref{eq:chain8}, gives
\begin{align*}
    \Bigl\Vert \sum_{i=1}^n f(x_i)\!\Bigl( l_{x_i} - \phi\!\Bigl(\tfrac{\cdot-x_i}{q_{X_n}}\Bigr)\Bigr) \Bigr\Vert_{L^2(\Omega)} \leq & \frac{\sqrt{2}C_{\text{\tiny{SZ}}} C_{\tau,\R^N}}{c_{\tau,\Omega}}\, h_{X_n,\Omega}^{\tau} q_{X_n}^{N/2-  \tau} \sqrt{n}    \|\phi\|_{W_2^{\tau}(\R^N)} \|f\|_{C_{\mathrm{ex}}(\Omega)} \\
    \leq & \frac{\sqrt{2}C_{\text{\tiny{SZ}}} C_{\tau,\R^N}}{c_{\tau,\Omega}}\, h_{X_n,\Omega}^{\tau} q_{X_n}^{N/2-  \tau} C_{\Omega,N,\rho}^{N/2} h_{X_n,\Omega}^{-N/2}    \|\phi\|_{W_2^{\tau}(\R^N)} \|f\|_{C_{\mathrm{ex}}(\Omega)}\\
    \leq &  \frac{\sqrt{2}C_{\text{\tiny{SZ}}} C_{\tau,\R^N}}{c_{\tau,\Omega}}\, \rho^{\tau-N/2} C_{\Omega,N,\rho}^{N/2} \|\phi\|_{W_2^{\tau}(\R^N)}  \|f\|_{C_{\mathrm{ex}}(\Omega)}.
\end{align*}
In the second inequality we used \eqref{eq:fillN}, and in the last inequality the mesh-ratio bound
\(\rho \ge h_{X_n,\Omega}/q_{X_n}\).

\medskip\noindent\textbf{(ii) The “bump’’ term.}  Since the supports of the bumps are disjoint, we have
\begin{align*}
\left\|\sum_{i=1}^n f(x_i)\,\phi(\tfrac{\cdot-x_i}{q_{X_n}})\right\|_{L^2(\Omega)}^2
= & \sum_{i=1}^n |f(x_i)|^2\,\|\phi(\tfrac{\cdot-x_i}{q_{X_n}})\|_{L^2(\Omega)}^2 \\
\le & n\,q_{X_n}^{\,N}\, \|\phi\|_{L^2(\mathbb{R}^N)}^2 \|f\|_{C_{\mathrm{ex}}(\Omega)}^2,\\
 \le & C_{\Omega,N,\rho}^N\,h_{X_n,\Omega}^{-N}\,q_{X_n}^N\,\|f\|_{C_{\mathrm{ex}}(\Omega)}^2 \\
\le & C_{\Omega,N,\rho}^N\,\,\|f\|_{C_{\mathrm{ex}}(\Omega)}^2,
\end{align*}
where we used \eqref{eq:fillN} and $\|\phi\|_{L^2(\mathbb{R}^N)} =1$, and in the last inequality the bound $q_{X_n} \le h_{X_n,\Omega}$.

\medskip\noindent\textbf{(iii) Conclusion.}  Combining (i) and (ii) in \eqref{eq:chain000} yields
\[
\|\Pi^n_{C,L^2}f\|_{L^2(\Omega)}
\;\le\;\left(\frac{\sqrt{2}C_{\text{\tiny{SZ}}} C_{\tau,\R^N}}{c_{\tau,\Omega}}\, \rho^{\tau-N/2} C_{\Omega,N,\rho}^{N/2} \|\phi\|_{W_2^{\tau}(\R^N)}+C_{\Omega,N,\rho}^{N/2}\right)\,\|f\|_{C_{\mathrm{ex}}(\Omega)},
\]
and hence 
\(\|\Pi^n_{C,L^2}\|_{\mathcal L(C_{\mathrm{ex}}(\Omega),L^2(\Omega))}\) is bounded independently of \(n\).
\end{proof}
\noindent Combining the uniform boundedness estimate from Theorem~\ref{th:CLuniOp} with the admissibility statement in Proposition~\ref{pro:adCL} verifies the hypotheses of the abstract “escaping the native space’’ principle, Theorem~\ref{theo:conncentionLebeEscaping}. As a consequence, we obtain the following concrete \(L^{2}\)-convergence result for continuous target functions on $\overline{\Omega}$.

\begin{theorem}
\label{thm:escape_C_L2}
Let \(k\) be a Sobolev kernel of order \(\tau\) with \(\lfloor \tau \rfloor > N/2\), and let \(\Omega \subset \R^{N}\) be a bounded Lipschitz domain satisfying the  interior cone condition for the angle \(\eta \in (0,\pi/2)\) and radius \(R>0\). Let \((X_n)_{n\in\N}\) be a nested, quasi\mbox{-}uniform sequence of finite node sets \(X_n \subset \Omega\) with dense union in \(\Omega\), i.e.,
\[
  X_1 \subset X_2 \subset \cdots, 
  \qquad 
  \overline{\bigcup_{n\in\N} X_n} = \overline{\Omega},
\]
and assume that for some \(\rho \ge 1\) and all \(n\in\N\),
\[
  h_{X_n,\Omega} \le h_1(\tau,\eta)\,R,
  \qquad
  \frac{h_{X_n,\Omega}}{q_{X_n}} \le \rho,
\]
where \(h_1(\tau,\eta)\) is the constant from~\eqref{eq:interiorH}. Then, for every \(f \in C_{\mathrm{ex}}(\Omega)\),
\[
  \lim_{n\to\infty} \bigl\| f - \Pi_{C,L^2}^n f \bigr\|_{L^2(\Omega)} = 0 .
\]
\end{theorem}

\begin{proof}
By Proposition~\ref{pro:adCL}, the pair \((C_{\mathrm{ex}}(\Omega),L^{2}(\Omega))\) is \(k\)\mbox{-}admissible under the stated domain properties and sampling assumptions. Moreover, Theorem~\ref{th:CLuniOp} yields a bound on \(\|\Pi^n_{C,L^2}\|_{\mathcal{L}(C_{\mathrm{ex}}(\Omega),L^2(\Omega))}\) that is uniform in \(n\). Hence all hypotheses of Theorem~\ref{theo:conncentionLebeEscaping} are satisfied, and the claim follows.
\end{proof}

\noindent In particular, for Sobolev kernels and quasi\mbox{-}uniform node families with \(h_{X_n,\Omega}\to 0\), kernel interpolation convergences for every \(f\in C_{\mathrm{ex}}(\Omega)\) in the \(L^{2}(\Omega)\)-norm. Via the isometric identification \(C_{\mathrm{ex}}(\Omega)\cong C(\overline{\Omega})\), the same conclusion holds for all \(f\in C(\overline{\Omega})\) in the $L^2(\overline{\Omega})$-norm, since \(\partial\Omega\) has Lebesgue measure zero.\\
Under additional hypotheses (specified in the next section), the convergence can be strengthened to the supremum norm, yielding \(C_{\mathrm{ex}}(\Omega)\)\mbox{-}convergence of the interpolants.

\section{Approximating continuous functions by Sobolev kernel interpolation in the \(C\)-norm}\label{sec5}
\label{sec:Cnorm-approx}

Throughout this section we take
\[
  A(\Omega) \;=\; B(\Omega) \;=\; C_{\mathrm{ex}}(\Omega),
\]
where \(\Omega\subset\mathbb{R}^{N}\) is a bounded Lipschitz domain. We first verify that this choice yields a \(k\)-admissible pair in the sense of Definition~\ref{def:admis-closure} for Sobolev kernels of order \(\tau\) with \(\lfloor \tau \rfloor > \tfrac{N}{2}\).

\begin{proposition}\label{prop:ad_CC}
Let \(\Omega \subset \mathbb{R}^N\) be a bounded Lipschitz domain, and let \(k\) be a Sobolev kernel of order \(\tau\) with \(\lfloor \tau \rfloor>\tfrac{N}{2}\). Then \(\bigl(C_{\mathrm{ex}}(\Omega),C_{\mathrm{ex}}(\Omega)\bigr)\) is a \(k\)\mbox{-}admissible pair.
\end{proposition}
\begin{proof}
Items {(a)}-{(c)} of Definition \ref{def:admis-closure} are exactly as in Proposition~\ref{pro:adCL}, since 
\(A(\Omega)=C_{\mathrm{ex}}(\Omega)\) for both settings. For {(d)},
because \(A(\Omega)=B(\Omega)=C_{\mathrm{ex}}(\Omega)\), the
embedding \(A(\Omega)\hookrightarrow B(\Omega)\) is the
identity and holds with constant \(C_B=1\). Hence \((C_{\mathrm{ex}}(\Omega),
C_{\mathrm{ex}}(\Omega))\) is \(k\)-admissible.
\end{proof}
\noindent In the present setting,  the operator norm of the interpolation operator \eqref{eq:Pi_C_B} coincides with the classical Lebesgue constant, which we now introduce.
\begin{definition}[Lebesgue function and constant]\label{def:LebesgueConstant}
Let \(X_n=\{x_1,\dots,x_n\}\subset\Omega\) be pairwise distinct, \(k\) be s.p.d., and \(\{l_{x_i}\}_{i=1}^n\subset V(X_n)\) the Lagrange basis. The Lebesgue function and constant are
\[
  \Lambda_{X_n}(x) \;:=\; \sum_{i=1}^n \bigl|l_{x_i}(x)\bigr|, \quad x\in\Omega,
  \qquad
  \Lambda_{X_n} \;:=\; \bigl\|\Lambda_{X_n}\!(\,\cdot\,)\bigr\|_{C_{\mathrm{ex}}(\Omega)}.
\]
\end{definition}
\noindent With Definition~\ref{def:LebesgueConstant} in place, the following classical result holds (see \cite[p.~41]{powell1981}).
\begin{lemma}\label{lem:Lebesgue-operator}
For every finite \(X_n\subset\Omega\),
\[
  \bigl\|\Pi^n_{C,C}\bigr\|_{\mathcal{L}\left(C_{\mathrm{ex}}(\Omega)\right)} \;=\; \Lambda_{X_n}.
\]
\end{lemma}

\begin{proof}
For the upper bound, for any \(f\in C_{\mathrm{ex}}(\Omega)\),
\[
  \bigl\|\Pi^n_{C,C} f\bigr\|_{C_{\mathrm{ex}}(\Omega)}
  \;=\; \Bigl\|\sum_{i=1}^n f(x_i)\,l_{x_i}\Bigr\|_{C_{\mathrm{ex}}(\Omega)}
  \;\le\; \|f\|_{C_{\mathrm{ex}}(\Omega)}\,\bigl\|\Lambda_{X_n}\!(\,\cdot\,)\bigr\|_{C_{\mathrm{ex}}(\Omega)}
  \;=\; \Lambda_{X_n}\,\|f\|_{C_{\mathrm{ex}}(\Omega)}.
\]
Hence \(\|\Pi^n_{C,C}\|\le \Lambda_{X_n}\).
For the reverse inequality, choose \(x^\ast\in \overline \Omega\) with
\(
  \Lambda_{X_n}(x^\ast)=\Lambda_{X_n}
\)
and define data on the nodes by \(d_i:=\mathrm{sign}\bigl(l_{x_i}(x^\ast)\bigr)\in\{-1,0,1\}\). Then, there exists \(\tilde f\in C(\overline \Omega)\) with \(\|\tilde f|_\Omega \|_{C_{\mathrm{ex}}(\Omega)}= \|\tilde f\|_{ C(\overline \Omega)}\le 1\) and \(\tilde f(x_i)=d_i\) for all \(i=1,...,n\). Thus, 
\[
  \bigl|\Pi^n_{C,C}\tilde f(x^\ast)\bigr|
  \;=\; \Bigl|\sum_{i=1}^n \tilde f(x_i)\,l_{x_i}(x^\ast)\Bigr|
  \;=\; \sum_{i=1}^n \bigl|l_{x_i}(x^\ast)\bigr|
  \;=\; \Lambda_{X_n}.
\]
Therefore, \(\|\Pi^n_{C,C}\|\ge \Lambda_{X_n}\), and the claim follows.
\end{proof}
\noindent 
Consequently, by Theorem~\ref{theo:conncentionLebeEscaping}, “escaping the native space’’ in the supremum norm is ensured once a uniform upper bound for the Lebesgue constants \(\{\Lambda_{X_n}\}_{n\in\N}\) is available. For Sobolev-type kernels and quasi\mbox{-}uniform centers, De~Marchi {et~al.}\,\cite{marchi2010} derive explicit upper bounds  for \(\Lambda_{X_n}\) that grow proportionally \(n^{1/2}\), which does not yield uniform boundedness. On compact manifolds without boundary, Hangelbroek, Narcowich, and Ward~\cite{HNW10} obtain uniform boundedness under quasi\mbox{-}uniform sampling; the boundaryless setting is essential in their argument. In \cite{HNRW17} a domain-extension strategy with centers outside the domain $\Omega$ providing bounded Lebesgue constants on compact general sets is developed. Note that here the Lagrange functions are not the same as when taking just centers inside the domain. \\
\noindent In the sequel we present a direct argument showing that, when \(\Omega=(a,b)\subset\mathbb{R}\) is an interval, the Lebesgue constants \(\{\Lambda_{X_n}\}_{n\in\N}\) associated with quasi-uniform node sets are uniformly bounded for Sobolev kernels of order \(\tau\) with \(\lfloor \tau \rfloor>\tfrac{N}{2}\) satisfying the local norm decomposition from Definition~\ref{def:SObolevWIthLocalNOrm}. Similar results have also been established for other classes of kernels that are not of Sobolev type; see, for instance, \cite{Santin2023,Bos2008}. 
There, the Lagrange functions \(l_{x_i}\), \(i=1,\dots,n\), have compact support contained in the intervals \([x_{i-1},x_{i+1}]\), where the centers satisfy
$
  x_0:=a < x_1 < \cdots < x_n < b :=x_{n+1}.
$
As a consequence, at every point of \([a,b]\) at most two Lagrange functions are simultaneously nonzero. This strong locality property considerably simplifies the analysis of the associated Lebesgue function and, hence, of the Lebesgue constants.

\subsection{Bounding the Lebesgue constant of Sobolev kernels with local norm decomposition on intervals}
\label{subsec:Lebesgue-interval}
An interval \(\Omega=(a,b)\subset\R\) with \(-\infty<a<b<\infty\) is a bounded Lipschitz domain and satisfies the interior cone condition (Definition~\ref{def:interior}) for any angle \(\eta\in(0,\tfrac{\pi}{2})\) and radius \(R=\tfrac{b-a}{2}\).
Hence the hypotheses of Theorem~\ref{theo:EstSemiNormSobolev} hold whenever
\(h_{X_n,\Omega}\le h_1(\tau,\eta)\,R\) with \(h_1\) from \eqref{eq:interiorH}.
In particular, taking the limit from the left \(\eta \to \pi/2^-\) yields the explicit expression
\begin{equation*}
  \lim_{\eta\to\pi/2^-} h_1(\tau,\eta)
  \;=\; \frac{3\sqrt{7}}{16\tau^2(32+3\sqrt{7})}\,.
\end{equation*}
Moreover, since
\begin{equation*}
  \frac{3\sqrt{7}}{16(32+3\sqrt{7})} \;>\; \frac{1}{100},
\end{equation*}
a convenient sufficient condition for Theorem~\ref{theo:EstSemiNormSobolev} to hold is
\begin{equation}\label{eq:interval-small-h}
  h_{X_n,\Omega}
  \;\le\; \frac{b-a}{100\,\tau^2}.
\end{equation}
Furthermore, on an interval with $X_n:= \{x_1,...,x_n\}$ is ordered increasingly, that is $a<x_1< \cdots < x_n < b$, the fill distance satisfies
\begin{equation}\label{eq:simpFill}
  h_{X_n,\Omega}
  \;=\;
  \max\!\left\{\,x_1-a,\ \max_{1\le i\le n-1}\frac{x_{i+1}-x_i}{2},\ b-x_n\right\}.
\end{equation}
In preparation for the proof establishing an exponential decay of the Lagrange functions, we first introduce two auxiliary constructions, which we briefly review and discuss before stating the theorem. These standard tools are:
\begin{enumerate}
\item \emph{Uniform partitioning.}
For \(M\in\N\), set
\begin{equation}\label{eq:hUMCondition0}
  \xi_j \;:=\; a + j \, \Delta\xi,\quad j=0,1,\dots,M+1,
  \qquad
  \Delta\xi \;:=\; \frac{b-a}{M+1},
\end{equation}
and define subintervals \(\Omega_j:=(\xi_j,\xi_{j+1})\), \(j=0,\dots,M\).
Assume now that
\begin{equation}\label{eq:hUMCondition}
  h_{X_n,\Omega}\ \le\ \frac{b-a}{1200\,\tau^2}.
\end{equation}
Then, for the choice
\begin{equation} \label{eq:hUMCondition5}
  M := \Bigl\lfloor \frac{b-a}{600 \tau^2 \,h_{X_n,\Omega}}\Bigr\rfloor - 1,  
\end{equation}
we indeed have \(M \ge 1\), and for $\tau>\tfrac12$, it holds
\[
  M \ \le\ \frac{b-a}{600 \tau^2 \,h_{X_n,\Omega}} - 1
  \quad\Longrightarrow\quad
  \Delta \xi \ \ge\ 600 \tau^2\, h_{X_n,\Omega}\ \ge\ 150\, h_{X_n,\Omega}.
\]
Since the maximum gap between consecutive nodes is at most $2h_{X_n,\Omega}$ (cf.~\eqref{eq:simpFill}), each $\Omega_j$ contains at least
$\#(\Omega_j \cap X_n) = \lfloor \frac{\Delta \xi }{2\,h_{X_n,\Omega}} \rfloor - 1 \geq 74$ nodes. 
By the same reasoning, for every subinterval $I\subset(a,b)$ that contains at least one node we have
\begin{equation}\label{eq:fillform2}
  h_{X_n\cap I, I}\ \le\ 2\,h_{X_n,\Omega}.
\end{equation}
Applying this with $I=\Omega_j$ and using $\Delta \xi \ge 600 \tau^2 h_{X_n,\Omega}$ gives
\begin{equation}\label{eq:hUMCondition2}
  h_{X_n\cap\Omega_j, \Omega_j}
  \ \le\ 2\,h_{X_n,\Omega}
  \ \le\ \frac{\Delta \xi}{300\,\tau^2},
  \qquad j=0,\dots,M.
\end{equation}
Moreover, it follows directly from \eqref{eq:simpFill} by a simple case distinction
(consider separately the case where the removed node is adjacent to a boundary
and the case where it is not) that, on intervals, deleting a single center enlarges the local fill distance by at most a factor \(3\). Hence
\begin{align}\label{eq:hUMCondition3}
  h_{\left(X_n \setminus \{x_i\}\right)\cap\Omega_j,\ \Omega_j}
  \ \le\ 3\, h_{X_n\cap\Omega_j,\ \Omega_j}\ \le\ 6\,h_{X_n,\Omega}
  \ \le\ \frac{\Delta \xi}{100\,\tau^2},
  \qquad j=0,\dots,M.
\end{align}
In particular, the assumptions of Theorem~\ref{theo:EstSemiNormSobolev} hold on each $\Omega_j$ both for the restricted set $X_n\cap\Omega_j$ and for $\left(X_n\setminus\{x_i\}\right)\cap\Omega_j$, because the local fill distances \eqref{eq:hUMCondition2}–\eqref{eq:hUMCondition3} satisfy the sufficient smallness condition \eqref{eq:interval-small-h}. These estimates will be used below to prove the exponential decay of the associated Lagrange functions.

\item \emph{Smooth cutoff functions.}
Let
\begin{equation}\label{eq:cutOff}
  \mu(t)=
  \begin{cases}
  e^{-1/t}, & t>0,\\[2pt]
  0, & t\le 0,
  \end{cases}
  \qquad
  \chi(t)=\dfrac{\mu(1-t)}{\mu(t)+\mu(1-t)}\,,
  \qquad
  \psi_{j}(x)=\chi\!\left(\dfrac{x-\xi_j}{\,\xi_{j+1}-\xi_j\,}\right),
\end{equation}
for \(j=0,\dots,M\).
Then \(\psi_j\in C^\infty(\R)\) and
\[
  \psi_j(x)=1 \ \text{for } x\le \xi_j,\qquad
  \psi_j(x)=0 \ \text{for } x\ge \xi_{j+1},\qquad
  0<\psi_j(x)<1 \ \text{for } x\in(\xi_j,\xi_{j+1}).
\]
Moreover, for each \(m\in\N_0\) there exists \(C_{\chi,m}>0\), independent of \(j\) and \(\Delta\xi\), such that
\begin{equation}\label{eq:psi-deriv-bnd}
  \|\psi_j^{(m)}\|_{C(\R)}
  \ \le\ \frac{C_{\chi,m}}{(\Delta\xi)^{m}}\,,\qquad m=0,1,2,\dots
\end{equation}
\end{enumerate}

\medskip
\noindent
For Sobolev kernels that admit the local norm decomposition  from Definition~\ref{def:SObolevWIthLocalNOrm},
these ingredients yield the following exponential decay of Lagrange functions; the proof adapts the approach of \cite{HNW10}.

\begin{theorem}\label{thm:exp-decay-interval}
Let \(k\) be a Sobolev kernel of integer order \(\tau\in\N\) that admits a local norm decomposition (Definition~\ref{def:SObolevWIthLocalNOrm}). Let \(X_n=\{x_1<\dots<x_n\}\subset(a,b)\) be pairwise distinct and satisfy the fill-distance condition \eqref{eq:hUMCondition}. Suppose further that 
\begin{equation}\label{eq:rho-qu}
 q_{X_n}<1 \quad\text{and}\quad h_{X_n,\Omega}\ \le\ \rho\, q_{X_n}\quad\text{for some }\quad\rho\ge 1.
\end{equation}
Then there exist constants \(\nu>0\) and \(C_{\text{\tiny{ED}}}>0\), depending only on \(k\) (especially on \(\tau\) and the local norm decomposition constants), the interval $(a,b)$ and on \(\rho\),  such that for every Lagrange function \(l_{x_i}\) and all \(x\in (a,b)\),
\begin{equation}\label{eq:exp-decay-statement}
    |l_{x_i}(x)| \ \le\ C_{\text{\tiny{ED}}}\,\exp\!\Bigl(-\,\nu\,\frac{|x-x_i|}{h_{X_n,\Omega}}\Bigr).
\end{equation}
\end{theorem}
\begin{proof}
Let \(M\in\N\) be chosen by \eqref{eq:hUMCondition5} and \(\{\xi_j\}_{j=0}^{M+1}\) as in \eqref{eq:hUMCondition0}.
Fix \(i\in\{1,\dots,n\}\) and set
\[
  J_i := \max\{\, j\in\{0,\dots,M\} : \xi_j \le x_i \,\}\quad\Longrightarrow\quad x_i\in[\xi_{J_i},\xi_{J_i+1}).
\]

\medskip\noindent
\textbf{Step 1: One-step tail contraction.}
Assume \(J_i<M\) and fix \(j\in\{J_i+1,\dots,M\}\) so that \(x_i<\xi_j\).
Using  the smooth cutoff from \(\psi_{j}\) from
\eqref{eq:cutOff}, we have for every node \(x_\ell\in X_n\):
\[
(\psi_{j} l_{x_i})(x_\ell)=
\begin{cases}
\psi_{j}(x_\ell) \cdot l_{x_i}(x_\ell)=1 \cdot \delta_{i\ell} = \delta_{i\ell}, & x_\ell\le \xi_{j},\\
\psi_{j}(x_\ell) \cdot l_{x_i}(x_\ell)=\psi_{j}(x_\ell) \cdot 0 = 0 =  \delta_{i\ell}, & x_\ell> \xi_{j}
\end{cases}
\]
Thus \(\psi_j l_{x_i}\) interpolates the same data as \(l_{x_i}\).
Further, by the Leibniz rule and \eqref{eq:psi-deriv-bnd}, for any
\(m\in\{0,\dots,\tau\}\),
$$\Vert (\psi_{j} l_{x_i})^{(m)} \Vert_{L^2(\R)} \leq \max_{r=0,...,m}\binom{m}{r}\frac{C_{\chi,r}}{(\Delta\xi)^{r}} \Vert  l_{x_i}  \Vert_{W_2^{\tau}(\R)}<\infty. $$
 Consequently, $\psi_{j} l_{x_i} \in W_2^{\tau}(\R)\simeq\mathcal{H}_k(\R)$.
By the native-space minimal-norm property of kernel interpolation, it follows 
\(\|l_{x_i}\|_{\mathcal{H}_k(\R)}\le \|\psi_{j} l_{x_i}\|_{\mathcal{H}_k(\R)}\).
Using the additive local norm decomposition
(Definition~\ref{def:SObolevWIthLocalNOrm}) with the partitions
\begin{align*}
\R= \overline{ (-\infty,\xi_{j})\cup (\xi_{j}, \infty)} \quad\text{and}\quad
\R= \overline{ (-\infty,\xi_{j})\cup\Omega_{j}\cup(\xi_{j+1},\infty)}
\end{align*}
then yields
\begin{align*}
  &\|l_{x_i}\|_{\mathcal{G}_k\left((-\infty,\xi_{j})\right)}^2
   +\|l_{x_i}\|_{\mathcal{G}_k\left((\xi_{j},\infty)\right)}^2
    \le 
   \underbrace{\|\psi_{j} l_{x_i}\|_{\mathcal{G}_k\left((-\infty,\xi_{j})\right)}^2}_{=\|l_{x_i}\|_{\mathcal{G}_k\left((-\infty,\xi_{j})\right)}^2}
   +\|\psi_{j} l_{x_i}\|_{\mathcal{G}_k\left(\Omega_{j}\right)}^2
   +\underbrace{
     \|\psi_{j} l_{x_i}\|_{\mathcal{G}_k\left((\xi_{j+1},\infty)\right)}^2}_{=0}
\end{align*}
by the  properties of the cutoff function \(\psi_{j}\). Canceling the common term gives
\[
  \|l_{x_i}\|_{\mathcal{G}_k\left((\xi_{j},\infty)\right)}^2
  \ \le\
  \|\psi_{j} l_{x_i}\|_{\mathcal{G}_k\left(\Omega_{j}\right)}^2.
\]
Invoking the uniform norm equivalence from Definition~\ref{def:SObolevWIthLocalNOrm} then results in
\begin{equation}\label{eq:loc-split}
  \|l_{x_i}\|_{W_2^{\tau}\left((\xi_{j},\infty)\right)}^2
  \ \le\ \Bigl(\frac{C_{\mathcal{G}_k}}{c_{\mathcal{G}_k}}\Bigr)^{\!2}\;
          \|\psi_{j} l_{x_i}\|_{W_2^{\tau}\left(\Omega_{j}\right)}^2.
\end{equation}
In one dimension, the Leibniz rule together with \eqref{eq:psi-deriv-bnd} gives
\begin{equation}\label{eq:prod}
  \|\psi_{j}\,l_{x_i}\|_{W_2^{\tau}\left(\Omega_{j}\right)}
  \;\le\; C_{1}\sum_{m=0}^{\tau} (\Delta\xi)^{-(\tau-m)}\,|l_{x_i}|_{W_2^{m}\left(\Omega_{j}\right)},
\end{equation}
with \(C_{1}=\tau \max_{0\le m\le\tau}\binom{\tau}{m}C_{\chi,m}\).
Since \(x_i \notin \Omega_{j}\), the cardinal property
\(l_{x_i}(\tilde{x})=0\)  for all \(\tilde{x} \in X_n\cap\Omega_{j}\) holds and \eqref{eq:hUMCondition2} is satisfied under the sampling condition \eqref{eq:hUMCondition}. Therefore, Theorem~\ref{theo:EstSemiNormSobolev} applies for \(m=0,1,\dots,\tau-1\)
and yields
\[
  |l_{x_i}|_{W_2^{m}\left(\Omega_{j}\right)}
  \;\le\; C_{\text{\tiny{SZ}}}\,(\Delta\xi)^{\,\tau-m}\,
          |l_{x_i}|_{W_2^{\tau}\left(\Omega_{j}\right)},
\]
with $C_{\text{\tiny{SZ}}}$ depending only on $\tau$.
Substituting this into \eqref{eq:prod} and squaring gives
\begin{equation}\label{eq:cell-to-tau}
  \|\psi_{j}\,l_{x_i}\|_{W_2^{\tau}\left(\Omega_{j}\right)}^2
  \;\le\; C_2\;|l_{x_i}|_{W_2^{\tau}\left(\Omega_{j}\right)}^2
  \quad\text{with}\quad C_2:=C_{1}^2\,C_{\text{\tiny{SZ}}}^2\,(\tau+1)^2.
\end{equation}
Combining \eqref{eq:loc-split} and \eqref{eq:cell-to-tau} yields
\begin{equation}\label{eq:0001}
  \|l_{x_i}\|_{W_2^{\tau}\left((\xi_{j},\infty)\right)}^2
  \ \le\ C_3\;|l_{x_i}|_{W_2^{\tau}\left(\Omega_{j}\right)}^2
  \ \le\ C_3\;\|l_{x_i}\|_{W_2^{\tau}\left(\Omega_{j}\right)}^2
  \quad\text{with}\quad C_3:=\max \left \{ 2, \Bigl(\frac{C_{\mathcal{G}_k}}{c_{\mathcal{G}_k}}\Bigr)^{\!2}\,C_2 \right \}.
\end{equation}
In addition, we have
\begin{align*}
      \|l_{x_i}\|_{W_2^{\tau}\left((\xi_{j},\infty)\right)}^2
   = 
     \|l_{x_i}\|_{W_2^{\tau}\left(\Omega_{j}\right)}^2 +\|l_{x_i}\|_{W_2^{\tau}\left((\xi_{j+1},\infty)\right)}^2  
      \geq \frac{1}{C_3}  \|l_{x_i}\|_{W_2^{\tau}\left((\xi_{j},\infty)\right)}^2 +\|l_{x_i}\|_{W_2^{\tau}\left((\xi_{j+1},\infty)\right)}^2,
\end{align*}
where we used \eqref{eq:0001} for the later inequality. Therefore, we obtain
\begin{align}\label{eq:rhCon}
  \|l_{x_i}\|_{W_2^{\tau}\left((\xi_{j+1},\infty)\right)}^2
  \ \le\ \mu^2\;\|l_{x_i}\|_{W_2^{\tau}\left((\xi_{j},\infty)\right)}^2
 \quad\text{with}\quad
  \mu^2:=\frac{C_3-1}{C_3}\in(0,1).
\end{align}
An entirely analogous argument with \(\tilde\psi_j:=1-\psi_j\) gives the left-tail contraction
\begin{equation}\label{eq:left-tail-contraction}
   \|l_{x_i}\|_{W^{\tau}_2((-\infty,\,\xi_{j-1}))}
   \ \le\  \mu^2\;\|l_{x_i}\|_{W^{\tau}_2((-\infty,\,\xi_{j}))}
   \qquad (1\le j\le J_i),
\end{equation}
provided \(J_i>1\).

\medskip\noindent
\textbf{Step 2: Pointwise exponential bound.}
Fix \(x\in(a,b)\). Without loss of generality, assume \(x\ge x_i\) (the other case is analogous using \eqref{eq:left-tail-contraction}).
Set
\[
  J_x := \max\{\, j\in\{0,\dots,M\} : \xi_j \le x \,\}\quad\Longrightarrow\quad x\in[\xi_{J_x},\xi_{J_x+1}).
\]
Applying Theorem~\ref{theo:EstSemiNormSobolev} on \(\Omega_{J_x}\) and using \eqref{eq:hUMCondition3} gives
\begin{equation}\label{eq:pointwise-1D}
     |l_{x_i}(x)|
      \le\ \|l_{x_i}\|_{L^{\infty}(\Omega_{J_x})}
      \le\ C_{\text{\tiny{SZ}}}\,h_{ \left(X_n \setminus \{x_i \} \right)\cap\Omega_{J_x},\,\Omega_{J_x}}^{\,\tau-\frac12}\,
           \|l_{x_i}\|_{W_2^{\tau}(\Omega_{J_x})}
      \le\ C_{4}\,h_{X_n,\Omega}^{\,\tau-\frac12}\,
           \|l_{x_i}\|_{W_2^{\tau}(\Omega_{J_x})},
\end{equation}
with \(C_4:=6^{\tau-\frac12}C_{\text{\tiny{SZ}}}\). Next, we define $s_x:=J_x-J_i \geq 0$ and distinguish two cases:\\
\emph{Case 1:} \(s_x\in\{0,1\}\).
Then \(x\) lies in the same or the next cell to the right of \(x_i\), and \eqref{eq:pointwise-1D} yields
\begin{equation} \label{eq:pointwise-tail2}
  |l_{x_i}(x)|\ \le\ C_4\,h_{X_n,\Omega}^{\,\tau-\frac12}\;\|l_{x_i}\|_{W_2^{\tau}(\Omega_{J_x})}\le\ C_4\,h_{X_n,\Omega}^{\,\tau-\frac12}\;\|l_{x_i}\|_{W_2^{\tau}(\R)}\le\ C_4\,h_{X_n,\Omega}^{\,\tau-\frac12}\;\mu^{\,s_x-1}\|l_{x_i}\|_{W_2^{\tau}(\R)},
\end{equation}
where the latter inequality holds, since $\mu \in (0,1)$.\\
\emph{Case 2:} \(s_x:=J_x-J_i>1\). Note that in this case $J_i<M$, since $M>M-1\geq J_x-1>J_i$.
Iterating the right-tail contraction \eqref{eq:rhCon} for \(j=J_i+1,\dots,J_x-1\) yields
\[
  \|l_{x_i}\|_{W_2^{\tau}((\xi_{J_x},\infty))}
    \le\ \mu^{} \|l_{x_i}\|_{W_2^{\tau}((\xi_{J_x-1},\infty))} \le  \dots   \le  \mu^{\,s_x-1}
       \|l_{x_i}\|_{W_2^{\tau}((\xi_{J_i+1},\infty))}.
\]
By monotonicity of norms,
\(
\|l_{x_i}\|_{W_2^{\tau}(\Omega_{J_x})}
 \le
\|l_{x_i}\|_{W_2^{\tau}((\xi_{J_x},\infty))}.
\)
Combining this with \eqref{eq:pointwise-1D} gives
\begin{equation}\label{eq:pointwise-tail1}
  |l_{x_i}(x)|
  \ \le\ C_4\,h_{X_n,\Omega}^{\,\tau-\frac12}\;
              \mu^{\,s_x-1}\;
             \|l_{x_i}\|_{W_2^{\tau}((\xi_{J_i+1},\infty))}\ \le\ C_4\,h_{X_n,\Omega}^{\,\tau-\frac12}\;
              \mu^{\,s_x-1}\;
             \|l_{x_i}\|_{W_2^{\tau}(\R)}.
\end{equation}
 So in both cases, the same upper bound holds.\\
Using the norm equivalence between \(\mathcal{H}_k(\R)\) and \(W_2^{\tau}(\R)\), the minimal-norm property of $l_{x_i}$ and  the standard bump \(\phi\) from \eqref{eq:standartBump}, gives
\[
  \|l_{x_i}\|_{W_2^{\tau}(\R)}
  \ \le\ \frac{1}{c_{\tau,\R}}\,\|l_{x_i}\|_{\mathcal{H}_k(\R)}
  \ \le\ \frac{1}{c_{\tau,\R}}\left\|\phi\!\left(\frac{\cdot-x_i}{q_{X_n}}\right)\right\|_{\mathcal{H}_k(\R)}
  \ \le\ \frac{C_{\tau,\R}}{c_{\tau,\R}}\,
         q_{X_n}^{\frac12-\tau}\,\|\phi\|_{W_2^{\tau}(\R)}.
\]
Here, for the last inequalities, we argue as in the derivation of \eqref{eq:chain11}.
Combining the latter inequality with \eqref{eq:pointwise-tail2}--\eqref{eq:pointwise-tail1} and invoking \eqref{eq:rho-qu} yields
\[
  |l_{x_i}(x)|
  \ \le\ C_5\,\mu^{s_x},\qquad
  C_5:= C_4 \mu^{-1}\,\frac{C_{\tau,\R}}{c_{\tau,\R}}\,
        \rho^{\,\tau-\frac12}\,\|\phi\|_{W_2^{\tau}(\R)}.
\]
Finally, using \(\Delta\xi \leq 1200 \tau^2 h_{X_n,\Omega}\), which results  from  \eqref{eq:hUMCondition5} via
\[
\Delta\xi = \frac{b-a}{M+1} = \frac{b-a}{\Bigl\lfloor \frac{b-a}{600 \tau^2 \,h_{X_n,\Omega}}\Bigr\rfloor} \leq  \frac{b-a}{ \frac{1}{2} \frac{b-a}{600 \tau^2 \,h_{X_n,\Omega}} } = 1200 \tau^2 h_{X_n,\Omega},
\]
it follows
\[
  s_x = J_x-J_i
  \ \ge\ \left\lfloor \frac{x-x_i}{\Delta \xi} \right\rfloor
  \ \ge\ \frac{x-x_i}{\Delta \xi}-1
  \ \ge\ \frac{x-x_i}{1200 \tau^2h_{X_n,\Omega}}-1.
\]
Writing \(\tilde\nu:=-\log\mu>0\) gives
\[
  |l_{x_i}(x)| \ \le\ C_5\,\mu^{s_x}
  \ =\ C_5\,e^{-\tilde\nu s_x}
  \ \le\ C_{\text{\tiny{ED}}}\,\exp\!\Bigl(-\,\nu\,\frac{|x-x_i|}{h_{X_n,\Omega}}\Bigr),
\]
with \(C_{\text{\tiny{ED}}}:=C_5\,e^{\tilde\nu}\) and \(\nu:=\frac{\tilde\nu}{1200 \tau^2}\).
The case \(x<x_i\) follows by the left-tail contraction, completing the proof.
\end{proof}

\noindent With the exponential decay of the Lagrange functions from Theorem~\ref{thm:exp-decay-interval},
we obtain the following uniform bound on the Lebesgue constant.

\begin{theorem}\label{thm:Lebesgue-uniform-interval}
Let the assumptions of Theorem~\ref{thm:exp-decay-interval} be satisfied on the interval $\Omega=(a,b)$.
Then there exists a constant $C_{\Lambda}>0$, depending only on the kernel
order $\tau$ (via the constants in Theorem~\ref{thm:exp-decay-interval}) and on the mesh ratio
$\rho$, such that
\[
  \Lambda_{X_n} \;\le\; C_{\Lambda}\,.
\]
\end{theorem}

\begin{proof}
Fix $x\in(a,b)$ and, for $s=0,1,2,\dots$, define the (two-sided) layers of width $h_{X_n,\Omega}$,
\[
  I_s
  := \Bigl([x-(s{+}1)h_{X_n,\Omega},\,x-sh_{X_n,\Omega}]
      \cup [x+sh_{X_n,\Omega},\,x+(s{+}1)h_{X_n,\Omega}]\Bigr)\cap(a,b),
\]
and the corresponding node subsets $S_s := X_n\cap I_s$.
Since the separation distance satisfies $\min_{i\neq j}|x_i-x_j| \ge 2q_{X_n}$, any interval
of length $L$ contains at most $1+\frac{L}{2q_{X_n}}$ nodes. Hence, for each $s\ge 0$,
\[
  \#S_s \;\le\; \Bigl(1+\frac{h_{X_n,\Omega}}{2q_{X_n}}\Bigr)
               + \Bigl(1+\frac{h_{X_n,\Omega}}{2q_{X_n}}\Bigr)
          \;\le\; 2+\frac{h_{X_n,\Omega}}{q_{X_n}}
          \;\le\; 2+\rho \;=: C_\rho .
\]
By Theorem~\ref{thm:exp-decay-interval}, there exist $C_{\text{\tiny{ED}}}>0$ and $\nu>0$,
independent of $n$ and $X_n$, such that
\[
  | l_{y}(x)| \;\le\; C_{\text{\tiny{ED}}}\,\exp\!\Bigl(-\nu\,\tfrac{|x-y|}{h_{X_n,\Omega}}\Bigr)
  \qquad \text{for all } y\in X_n.
\]
For $y\in S_s$ we have $|x-y|\ge s\,h_{X_n,\Omega}$, hence $| l_y(x)|\le C_{\text{\tiny{ED}}}e^{-\nu s}$.
Therefore,
\[
  \Lambda_{X_n}(x)
  \;=\; \sum_{y\in X_n} | l_y(x)|
  \;=\; \sum_{s=0}^{\infty}\ \sum_{y\in S_s} |  l_y(x)|
  \;\le\; \sum_{s=0}^{\infty} \#S_s \,\max_{y\in S_s}| l_y(x)|
  \;\le\; C_\rho\,C_{\text{\tiny{ED}}}\sum_{s=0}^{\infty} e^{-\nu s}
  \;=\; \frac{C_\rho\,C_{\text{\tiny{ED}}}}{1-e^{-\nu}}.
\]
The right-hand side is independent of $x$, $n$, and $X_n$. Taking the supremum over $x\in\Omega$
yields the statement with
\[
  C_{\Lambda} := \frac{C_\rho\,C_{\text{\tiny{ED}}}}{1-e^{-\nu}}\,,
\]
which depends only on $\tau$ (via $C_{\text{\tiny{ED}}},\nu$) and on $\rho$.
\end{proof}
\noindent As an immediate consequence of the preceding theorem, we obtain an
“escaping the native space’’ result in the supremum norm.
\begin{theorem}\label{thm:escape_C_C_interval}
Let $\Omega=(a,b)\subset\R$ and let $k$ be a Sobolev kernel of integer order
$\tau\in\N$ that admits a local norm decomposition {(}Definition~\ref{def:SObolevWIthLocalNOrm}{)}.
Let $(X_n)_{n\in\N}$ be a nested, quasi\mbox{-}uniform sequence of finite node sets
$X_n\subset\Omega$ with dense union in $\Omega$, i.e.,
\[
  X_1\subset X_2\subset\cdots,\qquad
  \overline{\bigcup_{n\in\N} X_n}=\overline{\Omega}.
\]
Assume that for some $\rho\ge1$ and all $n\in\N$,
\[
  h_{X_n,\Omega} \;\le\;  \frac{b-a}{1200\,\tau^2},
  \qquad   q_{X_n} \;<\;  1, \qquad
  \frac{h_{X_n,\Omega}}{q_{X_n}} \;\le\; \rho.
\]
Then, for every $f\in C_{\mathrm{ex}}((a,b))$,
\[
  \lim_{n\to\infty}\bigl\|f-\Pi^n_{C,C}f\bigr\|_{C_{\mathrm{ex}}((a,b))}=0.
\]
\end{theorem}

\begin{proof}
By Lemma~\ref{lem:Lebesgue-operator},
\(
  \|\Pi^n_{C,C}\|_{\mathcal{L}\left(C_{\mathrm{ex}}(\Omega)\right)}
  = \Lambda_{X_n}.
\)
Theorem~\ref{thm:Lebesgue-uniform-interval} yields a bound
$\sup_n \Lambda_{X_n}<\infty$ under the stated sampling assumptions.
Moreover, by Proposition~\ref{prop:ad_CC}, the pair
$\bigl(C_{\mathrm{ex}}(\Omega),C_{\mathrm{ex}}(\Omega)\bigr)$ is $k$-admissible.
Hence all hypotheses of Theorem~\ref{theo:conncentionLebeEscaping} are satisfied,
and the claim follows.
\end{proof}
\noindent Note that, in the one-dimensional case, the Mat\'ern kernels \eqref{eq:matern:12}--\eqref{eq:matern:52} have integer order $\tau$. Thus, as discussed at the end of Section~\ref{subsec:ti-algebraic}, these kernels admit the  local norm decomposition from Definition~\ref{def:SObolevWIthLocalNOrm}. Consequently, by the preceding theorem, kernel interpolation with these kernels yields approximants that converge to every $f \in C([a,b])$ in the supremum norm on $[a,b]$. In particular, unbounded oscillations of the interpolants under quasi-uniform refinement of the interpolation nodes are precluded.

\section{Conclusion and outlook}\label{sec6}
We have introduced a general criterion for {native-space escape by interpolation}: given Banach spaces $A(\Omega)$ (target class) and $B(\Omega)$ (error norm), the behavior of the extended interpolation operators
\[
  \Pi^{\,n}_{A,B}:A(\Omega)\to B(\Omega)
\]
governs whether interpolation converges in $\|\cdot\|_{B(\Omega)}$ for all $f\in A(\Omega)$. In particular, Theorem~\ref{theo:conncentionLebeEscaping} shows that {uniform boundedness} of $\{\Pi^{\,n}_{A,B}\}_n$ is necessary and sufficient for convergence on $A(\Omega)$. We applied this framework to Sobolev kernels with $\lfloor\tau\rfloor>N/2$ and continuous target functions on $\overline{\Omega}$, obtaining (i) $L^2(\Omega)$-convergence under quasi-uniform sampling, and (ii) $C_{\mathrm{ex}}(\Omega)$-convergence on intervals, the latter via uniform bounds on the Lebesgue constants $\Lambda_{X_n}$.

\medskip
\noindent
A natural direction is to investigate {infinitely smooth} kernels, such as the Gaussian. Our numerical experiments (see Fig.~\ref{fig:lebesgue-vs-n}), together with the evidence reported in \cite{marchi2010}, suggest the following conjecture: even for quasi-uniform node sets $X_n$ whose fill distance tends to zero, the Lebesgue constants $\Lambda_{X_n}$ for Gaussian interpolation are {not} uniformly bounded. If true, Theorem~\ref{theo:conncentionLebeEscaping} implies the existence of a continuous function $f\in C_{\mathrm{ex}}(\Omega)$ that is {not} approximated in $C_{\mathrm{ex}}(\Omega)$ by Gaussian interpolation along a nested, quasi-uniform sequence $X_n$ with dense union. Two concrete problems arise:
(i) Construct an explicit $f\in C_{\mathrm{ex}}(\Omega)$ witnessing non-convergence along some nested quasi-uniform sequence $X_n$.
(ii) Prove a quantitative {lower bound} for $\Lambda_{X_n}$ that grows with $n$ for the Gaussian (and, more broadly, for infinitely smooth translation-invariant kernels).

\begin{figure}[htbp]  
    \centering         
    \includegraphics[width=0.9\textwidth]{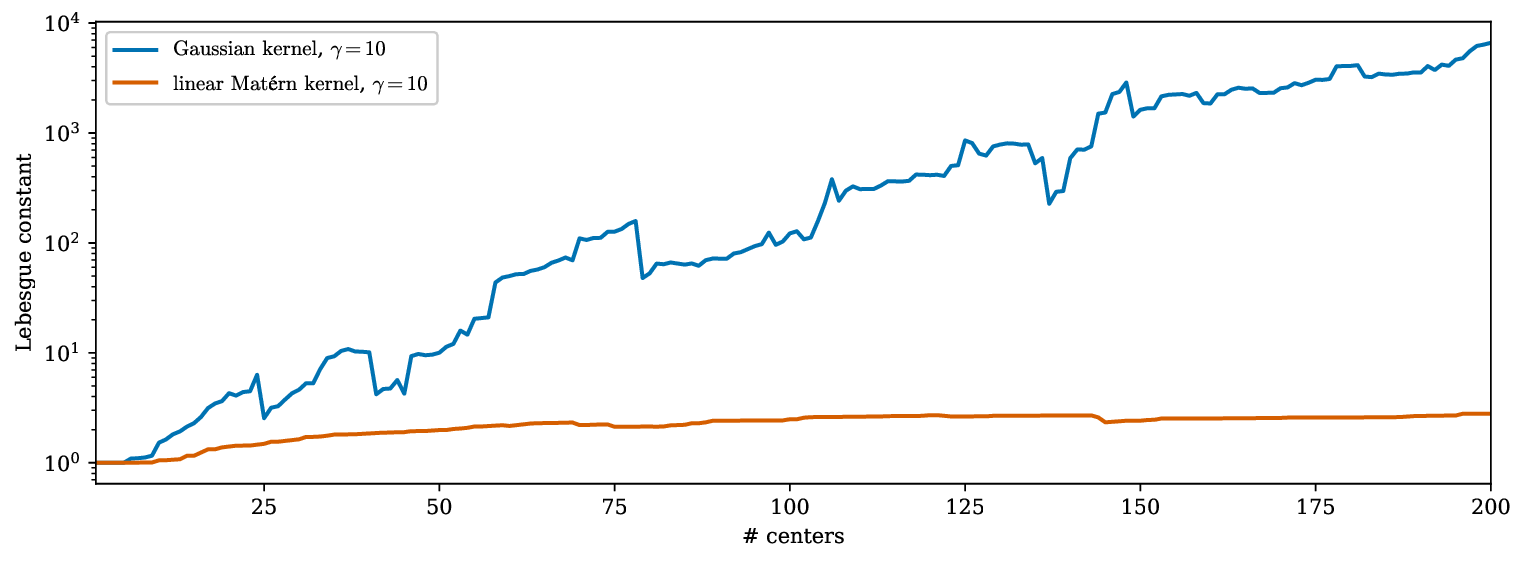}  
\caption{Lebesgue constant $\Lambda_{X_n}$ as a function of the number $n$ of
    quasi-uniformly distributed centers on the unit square for the Gaussian kernel with shape
    parameter $\gamma = 10$ and the linear Mat\'ern kernel~\eqref{eq:matern:32} with the
    same shape parameter. To generate a sequence of quasi-uniform centers, the geometric greedy
    algorithm was employed; see~\cite{DeMarchi2005}.}
  \label{fig:lebesgue-vs-n}
\end{figure}

\medskip
\noindent
For Sobolev kernels on intervals, we established $\sup_n\Lambda_{X_n}<\infty$. The key device is a smooth cutoff whose transition layers $(\xi_j,\xi_{j+1})$ lie strictly inside the domain and contain interpolation nodes, which permits invoking the scattered-zeros inequality (Theorem~\ref{theo:EstSemiNormSobolev}) on each subinterval. Extending this argument beyond one dimension is delicate: for example, for hypercubes, the transition regions of the smooth cutoff function can intersect the boundary, obstructing the use of Theorem \ref{theo:EstSemiNormSobolev}. Nevertheless, both our experience (see Fig.~\ref{fig:lebesgue-vs-n}) and the computations in \cite{marchi2010} indicate that on the unit square the Lebesgue constants for Sobolev kernels may still be uniformly bounded. This points to several avenues for future work, for extending our result to higher dimensions.

\section*{Acknowledgements}
  Funded by Deutsche Forschungsgemeinschaft (DFG, German Research Foundation) under Project No. 540080351 and Germany’s Excellence Strategy -- EXC 2075 -- 390740016. We acknowledge support from the Stuttgart Center for Simulation Science (SimTech).

   \end{document}